\documentclass[preprint]{elsarticle}
\usepackage{adjustbox}

\usepackage{tabularx,lipsum,environ}
\makeatletter
\newcommand{\problemtitle}[1]{\gdef\@problemtitle{#1}}
\newcommand{\probleminput}[1]{\gdef\@probleminput{#1}}
\newcommand{\problemquestion}[1]{\gdef\@problemquestion{#1}}
\NewEnviron{comproblem}{
  \problemtitle{}\probleminput{}\problemquestion{}
  \BODY
  \par\addvspace{.5\baselineskip}
  \noindent
  \begin{tabularx}{\textwidth}{@{\hspace{\parindent}} l X c}
    \multicolumn{2}{@{\hspace{\parindent}}l}{\@problemtitle} \\
    \textbf{Instance:} & \@probleminput \\
    \textbf{Question:} & \@problemquestion
  \end{tabularx}
  \par\addvspace{.5\baselineskip}
}
\makeatother
\usepackage{multicol}
\usepackage[title]{appendix}
\makeatletter
\newcommand{\@chapapp}{\relax}%
\makeatother
\usepackage{float}
\usepackage{xcolor}
\usepackage{graphicx,url}
\usepackage{amsfonts}
\usepackage{amsthm}
\usepackage{amsmath}
\usepackage{amssymb}
\usepackage{diagbox}
\usepackage{subcaption}
\usepackage[utf8]{inputenc}
\newtheorem{theorem}{Theorem}

\newtheorem{fact}{Fact}

\newtheorem{lemma}{Lemma}
\newtheorem{definition}{Definition}
\usepackage[figuresright]{rotating}
\usepackage{colortbl}
\usepackage{placeins}

\newtheorem{conjecture}[theorem]{Conjecture}

\journal{Discrete Applied Mathematics VSI: DAM-GO XI}

\begin{document}
\begin{frontmatter}
\title{An infinite family of Type~1 fullerene nanodiscs\tnoteref{t2}}
\author[ufrj]{Mariana Cruz}
\ead{mmartins@cos.ufrj.br}
\author[ufg]{Diane Castonguay}
\ead{diane@inf.ufg.br}
\author[ufrj]{Celina de Figueiredo}
\ead{celina@cos.ufrj.br}
\author[uerj]{Diana Sasaki}
\ead{diana.sasaki@ime.uerj.br}
\address[ufrj]{Federal University of Rio de Janeiro, Rio de Janeiro, Brazil}
\address[ufg]{Federal University of Goiás, Goiânia, Brazil}
\address[uerj]{Rio de Janeiro State University, Rio de Janeiro, Brazil}

\tnotetext[t2]{Partially supported by CNPq, FAPERJ and CAPES.}
\begin{abstract}
A total coloring of a graph colors all its elements, vertices and edges, with no adjacency conflicts.
The Total Coloring Conjecture (TCC) is a sixty year old challenge, says that every graph admits a total coloring with at most maximum degree plus two colors, and many graph parameters have been studied in connection with its validity.
If a graph admits a total coloring with maximum degree plus one colors, then it is Type~1, whereas it is Type~2, in case it does not admit a total coloring with maximum degree plus one colors but it does satisfy the TCC.
Cavicchioli, Murgolo and Ruini proposed in 2003 the hunting for a Type~2 snark with girth at least 5.
Brinkmann, Preissmann and  Sasaki in 2015 conjectured that there is no Type~2 cubic graph with girth at least 5.

We investigate the total coloring of fullerene nanodiscs, a class of cubic planar graphs with girth~$5$ arising in Chemistry. 
We prove that the central layer of an arbitrary fullerene nanodisc is $4$-total colorable, a necessary condition for the nanodisc to be Type~1.
We extend the obtained $4$-total coloring to a $4$-total coloring of the whole nanodisc, when the radius satisfies $r = 5+3k$, providing an infinite family of Type~1 nanodiscs.
\end{abstract}
\end{frontmatter}

\section{The large girth Type~1 conjecture}
\label{s:intro}

A total coloring of a graph $G$ is a color assignment from set $E \cup V$, where $E$ denotes the set of edges and $V$ denotes the set of vertices of the graph, such that adjacent or incident elements have different colors. A $k$-total coloring of a graph $G$ is a total coloring that uses a set of $k$ colors, and the graph is $k$-total colorable if it admits a $k$-total coloring. 
The total chromatic number $\chi''(G)$ is the smallest natural $k$ for which $G$ is $k$-total colorable.
Behzad and Vizing~\cite{behzad:65,vizing:64} independently conjectured the Total Coloring Conjecture (TCC) that for any simple graph $G$, we have $\chi''(G) \leq \Delta (G) + 2$. If $\chi''(G) = \Delta(G) + 1$, then the graph is called Type~1; if $\chi''(G) = \Delta (G) + 2$, then the graph is called Type~2. The TCC has been settled for cubic graphs~\cite{vijayaditya1971total}, but for a general degree graph it remains open for sixty years. The total chromatic number of cycle graphs $C_n$, well-known in the literature due to Yap~\cite{yap1996}, is the foundation of our findings, since we consider the relation between the girth and the TCC.\\


\begin{theorem}[Yap, 1996~\cite{yap1996}]\label{3coloringcn}
Let $G$ be the cycle graph $C_n$. Then

$$
\chi''(G) = \begin{cases} 3, & \mbox{if } n \equiv 0\mod3; \\ 4, & \mbox{otherwise. } \end{cases}
$$

\end{theorem}

In 2003, Cavicchioli et al.~\cite{Cavicchioli2003} asked for a Type~2 snark with girth at least~5, where snarks are cyclically 4-edge-connected cubic graphs that do not allow a 3-edge-coloring defined in the search for counterexamples to the Four-Color Conjecture~\cite{Gardner1976}. 
Since then, for more than twenty years, researchers have been hunting with no success a Type~2 snark with girth at least 5~\cite{Campos2011,Dantas2023}.

In 2015, as neither Type~2 cubic graphs with girth at least 5 nor Type~2 snarks were known, Brinkmann et al.~\cite{Brinkmann} realized that the conjecture was taking two steps at once, and the two requirements of being a snark and having girth at least 5 should better be treated independently. In doing so, they gave constructions of Type~2 snarks with girth 3 and 4, and proposed the following conjecture.


 \begin{conjecture}[Brinkmann, Preissmann and Sasaki, $2015$~\cite{Brinkmann}]\label{conjecturegirth5}
There is no Type~2 cubic graph with girth at least 5.
    \end{conjecture}
Motivated by this conjecture, we investigate the total coloring problem considering  
cubic planar graphs with large girth that model chemical structures: the fullerene nanodiscs.

Section~\ref{s:fullerene} defines our target class of fullerene nanodiscs, Section~\ref{s:central} proves that the natural strategy of using Yap's total coloring for the cycles establishes that a proper subgraph is Type~1, 
Section~\ref{s:extend} discusses the impossibility of extending the obtained 4-total coloring to the whole nanodisc and defines a 4-total coloring for the smallest fullerene nanodisc,
Section~\ref{s:block} establishes a suitable decomposition, 
and Section~\ref{s:Type1} considers the 4-total colorings developed in Sections~\ref{s:central} and~\ref{s:extend}  to provide the claimed infinite family of Type~1 nanodiscs.

\section{The fullerene nanodiscs are cubic graphs with girth~5}
\label{s:fullerene}

Fullerene is a spherically shaped molecule built entirely from carbon atoms. The carbon atoms form rings, each of which is either a pentagon or a hexagon. 
Each atom has chemical bond with exactly three other atoms. 
The first fullerene, known as a ``Buckminsterfullerene'', was discovered in 1985, see~\cite{c60}.
Kroto, Curl, and Smalley were subsequently awarded the 1996 Nobel Prize in Chemistry for their roles in the discovery of buckminsterfullerene and the related class of molecules, the fullerenes.
Molecular graphs of fullerenes are called fullerene graphs. Fullerene graphs are cubic, 3-connected, planar graphs with only pentagonal and hexagonal faces. 
A face defined by a triangle or by a cycle of length larger than 6 is forbidden in a fullerene graph. 
It follows from the Euler's Polyhedron Formula that any fullerene graph has exactly 12 pentagonal faces. The total number $f$ of faces in a fullerene graph with $n$ vertices is equal to $n/2 + 2$. Therefore, the number of hexagonal faces is exactly $n/2 - 10$.
A survey on fullerene graphs is presented in~\cite{andova2016}.

We shall focus on Fullerene nanodiscs $D_{r}$, $r \geq 2$.
The considered circular planar embedding of $D_r$ has its faces arranged into layers, one layer next to the nearest previous layer starting from a hexagonal cover until we reach the other  hexagonal cover. 
The key property of fullerene nanodiscs is that the 12 pentagonal faces lie in the central layer.
The distance between the inner (outer) layer and the central layer is given by the radius parameter $r \geq 2$.
Figure~\ref{f:smallestD2D3} shows the smallest nanodiscs $D_{2}$ and $D_{3}$, where we highlighted with blue color in the central layer the 12 pentagonal faces.

\begin{figure}[ht]
\centering
\includegraphics[scale=0.60]{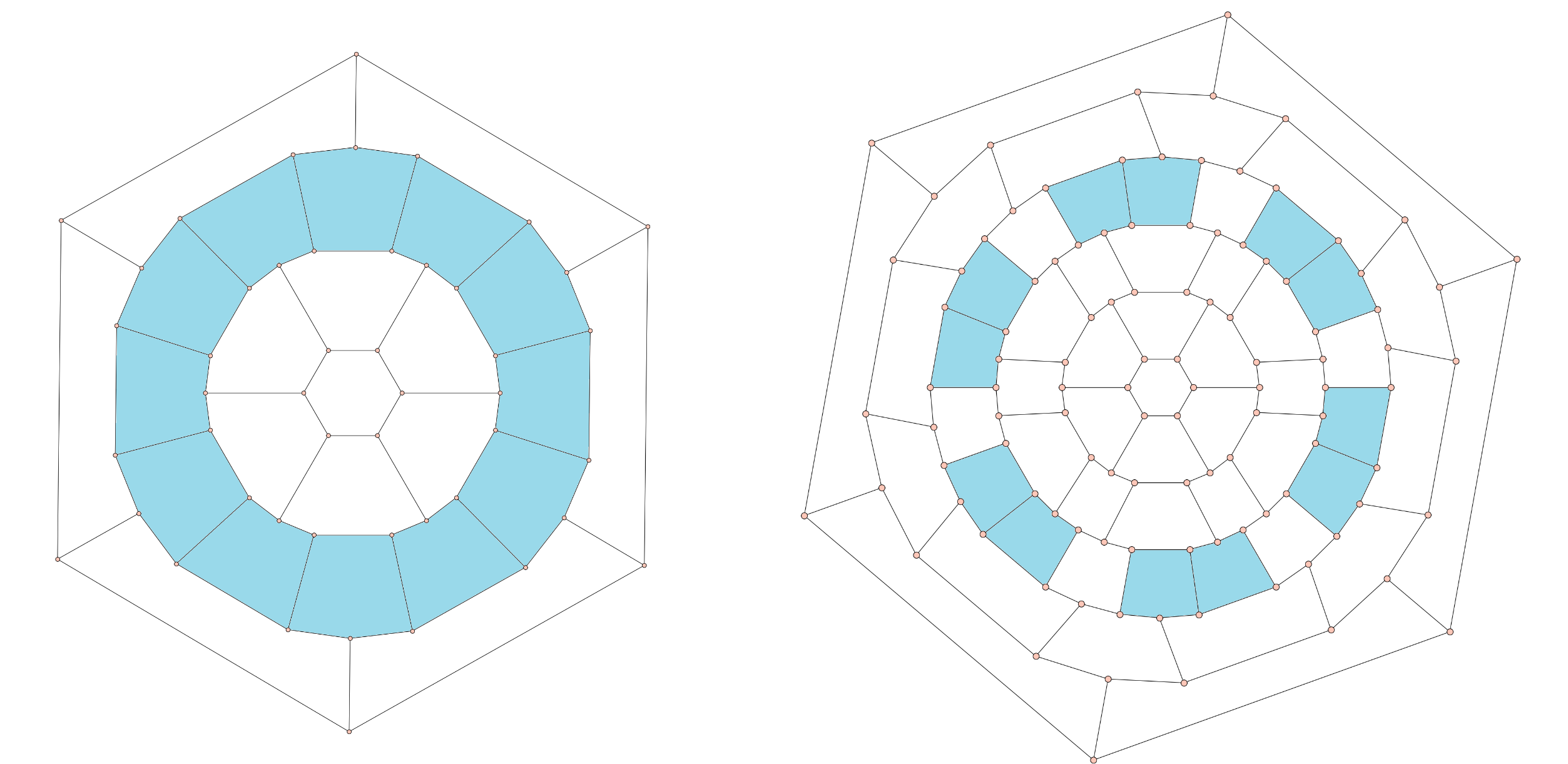}  
\caption{The smallest nanodiscs $D_{2}$ and $D_{3}$.}
\label{f:smallestD2D3}
\end{figure}
\FloatBarrier

The face sequence
$\left\{1,6,12,18,\ldots ,6(r-1),6r,6(r-1),\ldots ,18,12,6,1\right\}$  
provides the amount of faces on each layer of the nanodisc graph $D_{r}$. 
The layers with 6 faces are called inner and outer, respectively.
The unique layer with $6r$ faces is called central.
The $12$ pentagonal faces are distributed in the central layer among its $6r$ faces with the other $(6r-12)$ hexagonal faces.
The layers besides the central layer are called hexagonal layers.
The cycle sequence
$
\{C_6, C_{18}, \ldots, C_{12r-18},  C_{12r-6}, C_{12r-6}, C_{12r-18}, \ldots, C_{18}, C_6 \}
$  
provides the sizes of the auxiliary cycles that define the layers. 
The edges that connect auxiliary cycles are called radial edges and define a perfect matching of $D_r$.
A nanodisc $D_r$ contains 12$r^2$ vertices, 18$r^2$ edges, $2r+1$ layers, $2r$ auxiliary cycles and has girth~$5$. 

Since the number of vertices of every auxiliary cycle of a nanodisc is divisible by 3 and the radial edges define a perfect matching in $D_r$, a natural strategy to obtain a 4-total coloring is to use one color for all the radial edges and by Theorem~\ref{3coloringcn} use the other 3 colors in each auxiliary cycle. 
We shall see in Section~\ref{s:central} that the natural strategy gives a 4-total coloring of the central layer.
Unfortunately, as we shall see in Section~\ref{s:extend}, in order to extend this  4-total coloring to a 4-total coloring of the whole nanodisc, we will be forced to use all 4 colors in the other auxiliary cycles. 

\subsection{Some properties of the circular planar embedding}
\label{ss:planar}

We explicit some properties of the circular planar embedding that will be used in our proofs.
We say that a hexagon is balanced if it contains three vertices in each auxiliary cycle defining its layer.
The other faces of the planar embedding are unbalanced and contain just two vertices in one of the auxiliary cycles defining its layer.
The central layer is defined by the two auxiliary cycles $C_{12r-6}$ and $6r$ radial edges, and contains $6r$ faces, consisting of 12 pentagons and $6r - 12$ hexagons. 
Each pentagon has two vertices in one auxiliary cycle $C_{12r-6}$ and three vertices in the other auxiliary cycle $C_{12r-6}$. We say that two  pentagons are partitioned in the same way if they contain two vertices in the same auxiliary cycle $C_{12r-6}$.
Observe in Figure~\ref{f:smallestD2D3} the central layer and its 12 pentagons,
of the smallest nanodiscs $D_{2}$ and $D_{3}$.
Considering the circular embedding, when we traverse the central layer, say in clockwise, starting in a pentagonal face partitioned in a certain way, each time we meet the next pentagonal face, it is partitioned differently of the previous one. 
We say that two pentagons are consecutive, if when we traverse the central layer, say in clockwise, we meet no 
pentagonal face between them.
In particular, consecutive pentagons are partitioned differently.
We say that two faces are next to each other, if when we traverse the central layer, say in clockwise, we meet no face between them.
Observe in Figure~\ref{f:smallestD2D3} that every hexagon in the central layer of $D_{3}$ is balanced.

\begin{figure}[!htb]
    \centering
    \includegraphics[scale=0.15]{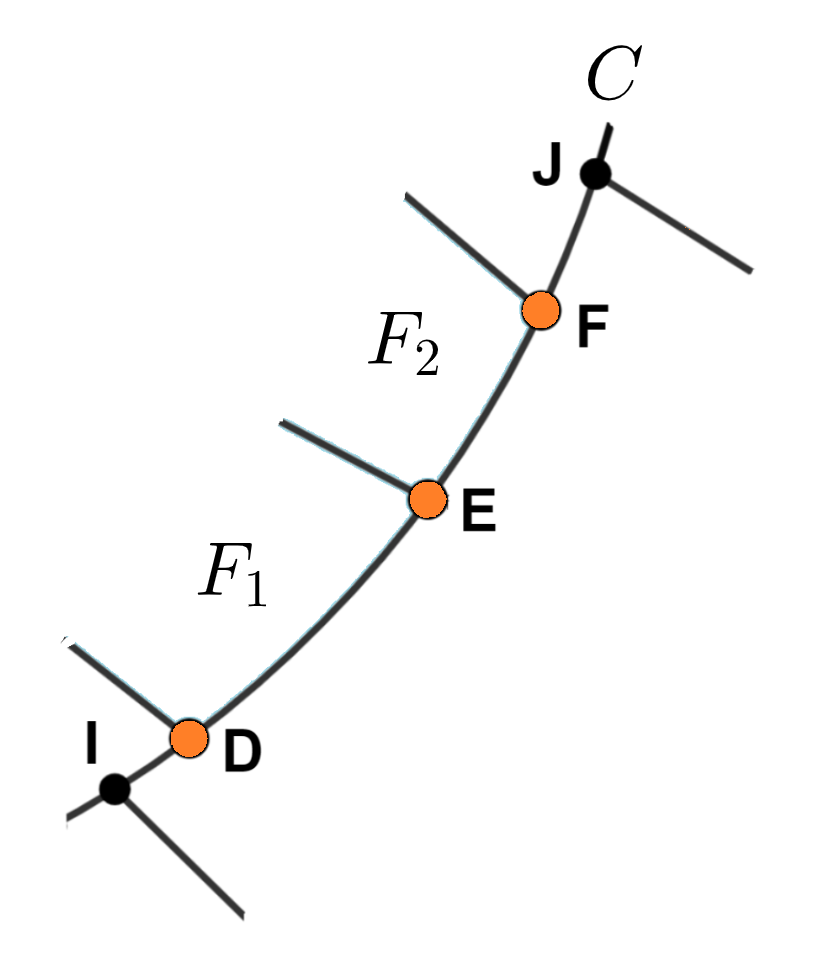}
    \caption{Vertices $I,D,E,F,J$ force the existence of a Large Forbidden Face.}
    \label{f:largeface}
\end{figure}
\FloatBarrier

\begin{lemma}[Large Forbidden Face]
\label{forbiddenfacelemma}
Let $F_1$ and $F_2$ be two faces next to each other
of layer $L$ 
containing more than 6 faces. Consider the two auxiliary cycles defining $L$. We cannot have the same auxiliary cycle containing precisely two vertices of $F_1$ and two vertices of $F_2$.
\end{lemma}

\begin{proof}

Please refer to Figure~\ref{f:largeface}, where we depict part of an auxiliary cycle $C$, that lies between two layers $L$ and $L'$,
where layer $L$ contains more than 6 faces. 
The faces $F_1$ and $F_2$ are next to each other
in layer $L$, and the vertices $D,E,F$ belong to a path in the auxiliary cycle $C$, such that $D$ and $E$ are of face $F_1$, $E$ and $F$ are of face $F_2$.
By considering vertices $I$ and $J$ adjacent respectively to $D$ and to $F$ in $C$, we find a path of five vertices $I,D,E,F,J$ in the cycle $C$, leading to a contradiction that $I,D,E,F,J$ must lie in a forbidden face in layer~$L'$.~\end{proof}


Observe in Figure~\ref{f:smallestD2D3} that $D_{2}$ is a degenerate case, since its central layer consists of the 12 pentagons. For larger radius $D_{r}$, r $\geq 3$, we
cannot have three pentagons next to each other in the central layer, as this will force in a layer next to the central layer two unbalanced hexagons next to each other partitioned in the same way, contradicting  the Large Forbidden Face lemma.


\begin{figure}[!hbt]
    \centering
    \includegraphics[scale=0.80]{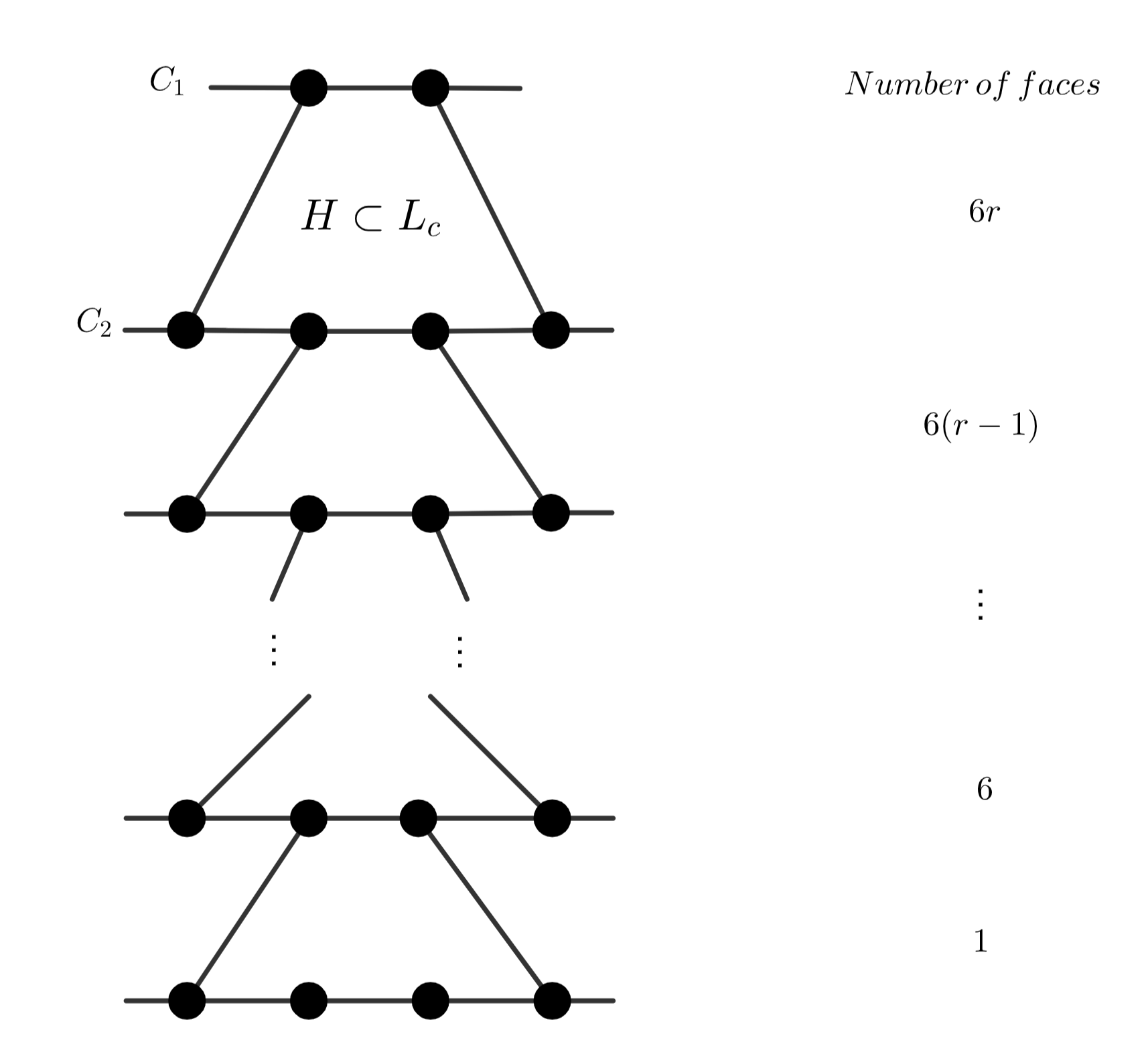}  
   \caption{Chain of unbalanced hexagons contained in a sequence of layers of $D_r$, $r \geq 3$ starting from the central layer $L_c$ with an unbalanced hexagon.}
    \label{f:lemmahexagonsbalanced}
\end{figure}
\FloatBarrier

\begin{lemma} [Balanced Hexagons]
\label{l:balanced}
    The $6r-12$ hexagons that belong to the central layer of a fullerene nanodisc $D_r$, $r \geq 3$, are balanced.
\end{lemma}

\begin{proof}
   Please refer to Figure~\ref{f:lemmahexagonsbalanced} to follow the argument of the proof. Let $D_r$, $r \geq 3$ be a fullerene nanodisc, $L_c$ its central layer with $6r$ faces and $C_1$ and $C_2$ the auxiliary cycles that define $L_c$. Suppose there is at least one unbalanced hexagon $H$ in $L_c$, and suppose without loss of generality that $H$ has 4 vertices in $C_2$. Since $D_r$ is cubic, there are two vertices of $H$ that lie in $C_2$ that have neighbors in the auxiliary cycle that composes the hexagonal layer with size $6(r-1)$, that is, these two vertices are extremes of radial edges that form a hexagonal face in the next layer with $6(r-1)$ faces. Note by the structure of the graph, this construction always generates an unbalanced hexagon in the adjacent next layer, having two vertices of degree two in an auxiliary cycle, which must be extremes of radial edges connecting to the next hexagonal layer. This implies that in the inner (outer) layer there will be a hexagon with two vertices of degree two in the cubic graph, a contradiction.
\end{proof}

So the central layer contains zero unbalanced hexagons, and when we traverse from the central layer to the inner (outer) hexagonal layer, each next hexagonal layer contains precisely 6 additional unbalanced hexagons equally distributed among the balanced hexagons.

\begin{figure}[!hbt]
    \centering
    \includegraphics[scale=0.80]{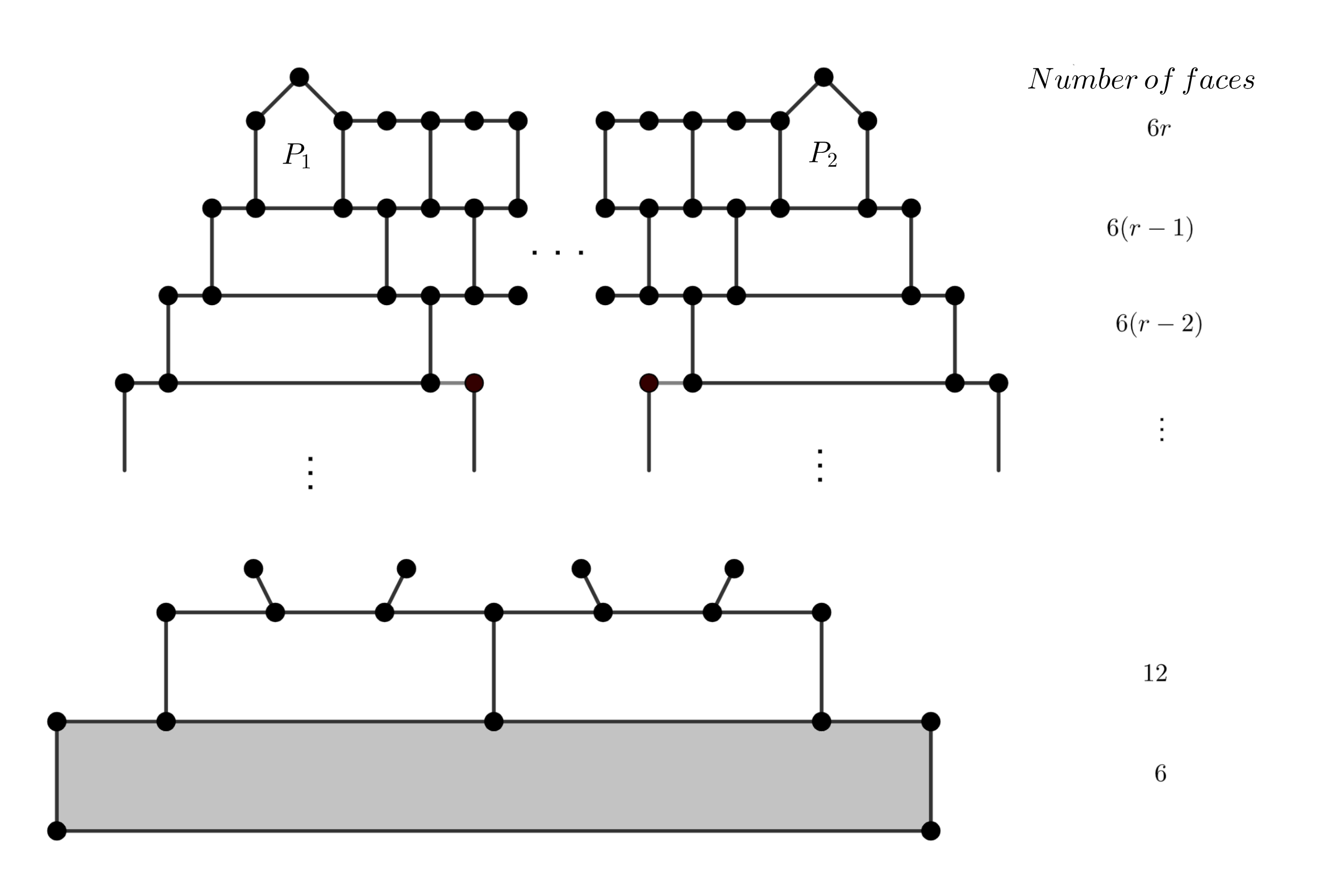}  
   \caption{Chain of unbalanced hexagons contained in consecutive layers of $D_r$, $r \geq 3$ starting from the central layer $L_c$ until we reach a Large Forbidden Face.}
    \label{f:nearestpentagonslemma}
\end{figure}
\FloatBarrier

\begin{lemma}[Alternating Pentagons]
\label{l:nearestpentagons}
In the central layer, each pair of consecutive pentagons are not partitioned in the same way.
\end{lemma}
\begin{proof}
Let $D_r$, $r \geq$ 2 be a fullerene nanodisc, $L_c$ its central layer with $6r$ faces, $C_1$ and $C_2$ the auxiliary cycles that define $L_c$, and an arbitrary pair of consecutive pentagons $P_1$ and $P_2$ that are partitioned in the same way.
By Lemma~\ref{forbiddenfacelemma}, $P_1$ and $P_2$ are not next to each other faces of $L_c$. 

Please refer to Figure~\ref{f:nearestpentagonslemma} to follow the construction of the proof. Suppose there are $h$ hexagons between $P_1$ and $P_2$, and notice that $h \leq r-2$. We can assume with no loss of generality that $P_1$ and $P_2$ have each two vertices in the auxiliary cycle $C_2$. By construction, these four vertices are not adjacent by radial edges to vertices of the hexagonal layer with $6(r-1)$ faces.
As a consequence, these vertices belong to unbalanced hexagons $H_1$ and $H_2$ in the hexagonal layer with $h-1$ balanced hexagons between $H_1$ and $H_2$. At each next hexagonal layer in the sequence of layers towards the inner (outer) layer, we have two unbalanced hexagons with one less balanced hexagons between them. Since $h \leq r-2$ and from the central layer to the inner (outer) layer there are $r$ layers, we obtain a Large Forbidden Face and a contradiction.
\end{proof}







\section{The central layer is always Type 1}
\label{s:central}


We prove that the natural strategy to use one color for all the radial edges and to use by Theorem~\ref{3coloringcn} the other 3 colors in each auxiliary cycle can be used to obtain a 4-total coloring of the central layer of an arbitrary nanodisc.


\begin{figure}[ht]
\centering
\includegraphics[scale=0.47]{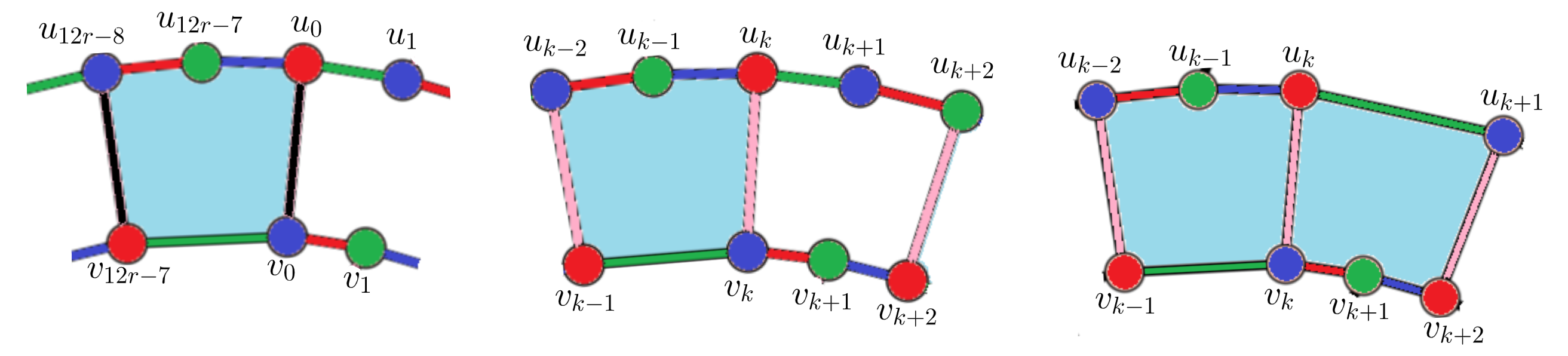}  
\caption{(a) Choice of $u_0$ and $v_0$ in $C_{12r-6}$ and $C_{12r-6}^{*}$, respectively; (b) Case 1; (c) Case 2.}
\label{f:lemageral}
\end{figure}
\FloatBarrier

\begin{lemma}\label{l:centrallayer4totalc}
The central layer of every fullerene nanodisc $D_r$, $r \geq 2$, is $4$-total colorable.
\end{lemma}

\begin{proof}
    Let $D_r$ be a nanodisc, $r \geq 2$, and $C_{12r-6}$ and $C_{12r-6}^{*}$ the two cycles that define its central layer. By Theorem~\ref{3coloringcn}, there is a $3$-total coloring for each cycle of size $12r-6$. Actually, the $3$-total coloring colors the elements of the cycle in circular order~\cite{yap1996}. We label the vertices of each cycle as follows.
    $$
V(C_{12r-6}) = \{u_{0},u_{1},...,u_{12r-7}\}, V(C_{12r-6}^{*}) = \{v_{0},v_{1},...,v_{12r-7}\}.
    $$

    Choose $u_0$ and $v_0$ such that these vertices are extremes of a pentagonal face, so that the edge $u_{0}v_{0}$ is the second radial edge of this pentagon in clockwise,
    as illustrated in Figure~\ref{f:lemageral}(a). We define the  3-total colorings for $C_{12r-6}$ by choosing color $1$ (red) for $u_0$ and for $C_{12r-6}^{*}$ by choosing color $3$ (blue) for~$v_0$. 
    Thus, the color assignment in $C_{12r-6}$ is such that for the vertices

$$
c(u_{k}) = \begin{cases} 1, & \mbox{if } k \equiv 0\mod3; \\ 3, & \mbox{if } k \equiv 1\mod3;\\ 2, & \mbox{if } k \equiv 2\mod3;  \end{cases}
$$
$$
$$ 
and for $C_{12r-6}^{*}$ is such that for the vertices

$$
c(v_{k}) = \begin{cases}  3, & \mbox{if } k \equiv 0\mod3;\\ 2, & \mbox{if } k \equiv 1\mod3; \\ 1, & \mbox{if } k \equiv 2\mod3;  \end{cases}
$$
$$
$$ 

and the color assignment in $C_{12r-6}$ and in $C_{12r-6}^{*}$ is such that for the edges
$$
c(u_{k}u_{k+1}) = c(v_{k+1}), \ \ \forall u_{k}u_{k+1} \in E(C_{12r-6}), 
$$
$$
c(v_{k}v_{k+1}) = c(u_k), \ \ \forall v_{k}v_{k+1} \in E(C_{12r-6}^{*}).
$$

Note that by Section~\ref{ss:planar}, the pentagonal and hexagonal structure in the central layer is well defined. 
There are at most two consecutive pentagons, with balanced hexagons between the  pentagons, and recall that the radial edges that connect the auxiliary cycles form a matching. Then, by choice of $u_{0}v_{0}$, counting of vertices and the structure of the faces in a central layer of $D_r$, we conclude that all radial edges that are a second edge in clockwise of a pentagon with 3 vertices in $C_{12r-6}$ are of type $u_{k}v_{k}$, and by symmetry, for pentagons with 3 vertices in $C_{12r-6}^{*}$, the edges of type $u_{k}v_{k}$ are the first in clockwise. 

Furthermore, by the structure of the pentagon, the first edge in clockwise of a pentagon is of type $u_kv_{k+1}$ (with 3 vertices in $C_{12r-6}$, and by symmetry, the second in clockwise when the pentagon has 3 vertices in $C_{12r-6}^{*}$), and the radial edges that belong to balanced hexagons can be of type $u_{k}v_{k}$ or $u_{k}v_{k+1}$, depending on the structure of the previous face. 
By the choice of $u_0$ and $v_0$ and the structure of the faces of the central layer, 
we show next that the radial edges occur in two ways: $u_{k}v_{k}$ or $u_{k}v_{k+1}$, 
and that we have no color conflict between the extremes of radial edges.
there are two cases to consider, up to symmetry:

\begin{itemize}

\item \textbf{Case 1:} There is a balanced hexagon next to the pentagon. Please refer to Figure~\ref{f:lemageral}(b). By the choice of the edge $u_{0}v_{0}$, the consecutive faces share an edge $u_{k}v_{k}$. As the hexagon is balanced, by counting vertices, the other edge of the hexagon will also be of the type $u_{k}v_{k}$. Note that by the coloring structure, vertices $u_{k} \in V(C_{12r-6})$ and $v_{k} \in V(C_{12r-6}^{*})$ are colored with different classes of colors. Thus, $c(u_{k}) \neq c(v_{k})$ and these edges do not imply in a color conflict. Note that the pentagon also has an edge of type $u_{k}v_{k+1}$. By the coloring structure given, note that $c(u_k) = c(v_{k}v_{k+1})$. This implies that $u_k$ and $v_{k+1}$ cannot belong to the same color class. Therefore, $c(u_{k}) \neq c(v_{k+1})$, and these type of edges do not imply in a color conflict.

 \item \textbf{Case 2:} There is a pentagon partitioned differently next to the pentagon. Please refer to Figure~\ref{f:lemageral}(c). 
 By the choice of the edge $u_{0}v_{0}$, the consecutive pentagonal faces share an edge $u_{k}v_{k}$. Note that in this case, the radial edges of the pentagon partitioned differently are $u_{k}v_{k}$ and $u_{k}v_{k+1}$. Thus, as seen by Case 1, these edges do not have colored extremes belonging to the same color class, and therefore we have no color conflict between the vertices of $C_{12r-6}$ and $C_{12r-6}^{*}$ that are extremes of radial edges.
\end{itemize}
As none of the radial edges implies in a color conflict, 
 we may give to all radial edges the same color $4$ (pink), and thus we obtain a $4$-total coloring of the central layer for $D_r$,~$r \geq 2$.   
\end{proof}

\section{The natural strategy does not extend}
\label{s:extend}

First, let us argue that for $D_2$ the natural strategy to use one color for all the radial edges and the other 3 colors for the vertices and edges of the auxiliary cycles is not possible to achieve a 4-total coloring.
Assume that the central layer is colored according to Section~\ref{s:central}.
So the 6 vertices of the central layer that are extremes of the radial edges of the hexagonal layer all receive the same color, say color $1$, among the 3 colors used in the auxiliary cycle of the central layer.
Since color $4$ must be given to all the radial edges of $D_2$, only two colors remain for the vertices of the inner (outer) $C_6$, and all edges of the inner (outer) $C_6$ are forced to receive the same color $1$.

Actually, we can establish a more general property, that says that for any $D_r$ the natural strategy to use one color for all the radial edges and the other 3 colors for the vertices and edges of the auxiliary cycles is not possible to achieve a 4-total coloring.

\begin{lemma}\label{l:notextend}
There is no 4-total coloring of $D_r$ where all radial edges receive the same color.
\end{lemma}
\begin{proof}
It suffices to prove that we cannot 4-total color the inner (outer) layer using one color $4$ for all the radial edges the other 3 colors for the vertices and edges of the auxiliary cycles.
The inner (outer) layer is defined by two cycles: a $C_6$ and a $C_{18}$.
Any such 4-total coloring gives the same color, say $c_1$, to the six vertices of the $C_{18}$ that are extremes of the radial edges of the inner (outer) layer, and the vertices of the $C_6$ must be colored with the remaining two colors $2$ and $3$. Therefore all edges of the inner (outer) $C_6$ are forced to receive the same color $1$.
\end{proof}

Lemma~\ref{l:notextend} implies that every 4-total coloring of the whole $D_2$ has to use all four colors in the two auxiliary cycles that compose the central layer.
We give in Figure~\ref{f:twototalD2} two  4-total colorings of $D_2$,
and discuss in the sequel how to define for larger radius a 4-total coloring that extends both a suitable 4-total coloring of $D_2$ and the 4-total coloring of the central layer obtained in Section~\ref{s:central}. 
Observe that the 4-total coloring of $D_2$ depicted on the right has all radial edges of the central layer colored with the same color and has a circular symmetry, once we traverse one third of the elements of the graph, the same pattern is repeated twice.

Next, in Section~\ref{s:block}, we define, for an arbitrary nanodisc, a nanodisc block that contains one third of the elements of the graph, and a suitable decomposition of the whole nanodisc.
In Section~\ref{s:Type1}, we are able to use such a decomposition to describe a 4-total coloring of the whole nanodisc, for an infinite family of nanodiscs with suitable values of $r$. 

\begin{figure}[!hbt]
    \centering
    \includegraphics[scale=0.59]{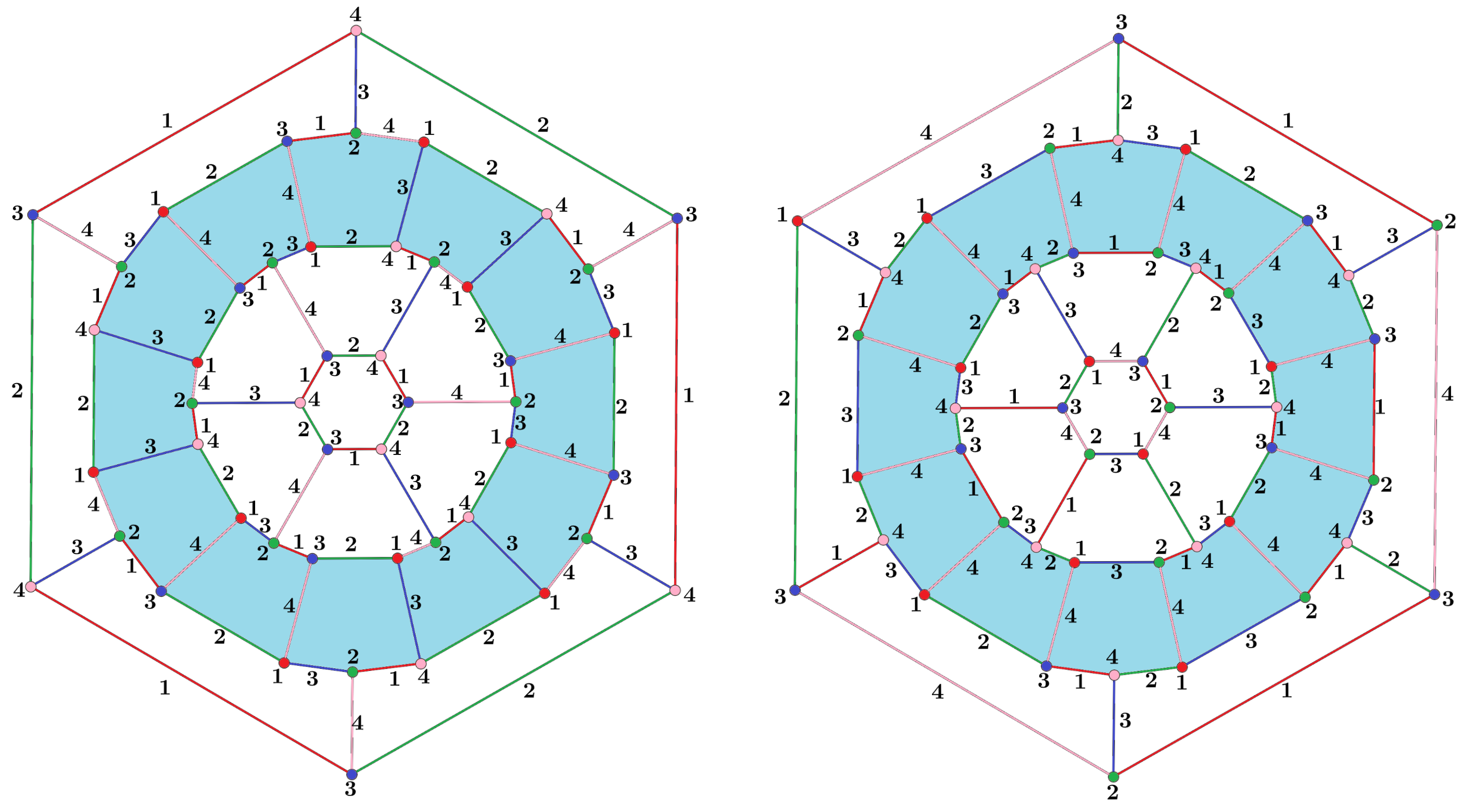}  
   \caption{Two 4-total colorings of $D_2$.}
    \label{f:twototalD2}
\end{figure}
\FloatBarrier

\section{Construction of nanodisc blocks}
\label{s:block}

The symmetrical construction of the fullerene nanodisc in relation to the central layer allows us to establish a parallel with cartography and divide the structure into two hemispheres as follows:

\begin{definition}[Outer and inner hemispheres]
Let $D_r, r \geq 2$, be a fullerene nanodisc and its auxiliary cycles sequence from the outer to the inner layer 
$
\{C_{6}, C_{18},\ldots , C_{12r-6}, C_{12r-6}, \ldots , C_{18}, C_{6}\}.
$
The \emph{outer hemisphere of $D_r$}, denoted by $O_{H}(D_r)$, is the cubic semigraph induced by vertices of the decreasing sequence of cycles from the outer layer to the central layer
$
\{C_{6},C_{18},\ldots C_{12r-6}\}.
$
Similarly, the \emph{inner hemisphere of $D_r$}, denoted by $I_{H}(D_r)$, is the cubic semigraph induced by vertices of the increasing sequence of auxiliary cycles from the central layer to the inner layer
$
\{C_{12r-6},\ldots , C_{18},C_{6}\}.
$

Each of the $6r$ radial edges $uv$ of the central layer  defines a semiedge of each hemisphere, $u{\cdot}$ for $O_{H}(D_r)$, and $v{\cdot}$ for $I_{H}(D_r)$. Note that $O_{H}$ and $I_{H}$  are isomorphic semigraphs. Figure~\ref{f:D3hemisphere} illustrates a nanodisc and its respective hemispheres highlighted.
\end{definition}

\begin{figure}[!hbt]
    \centering
    \includegraphics[scale=0.95]{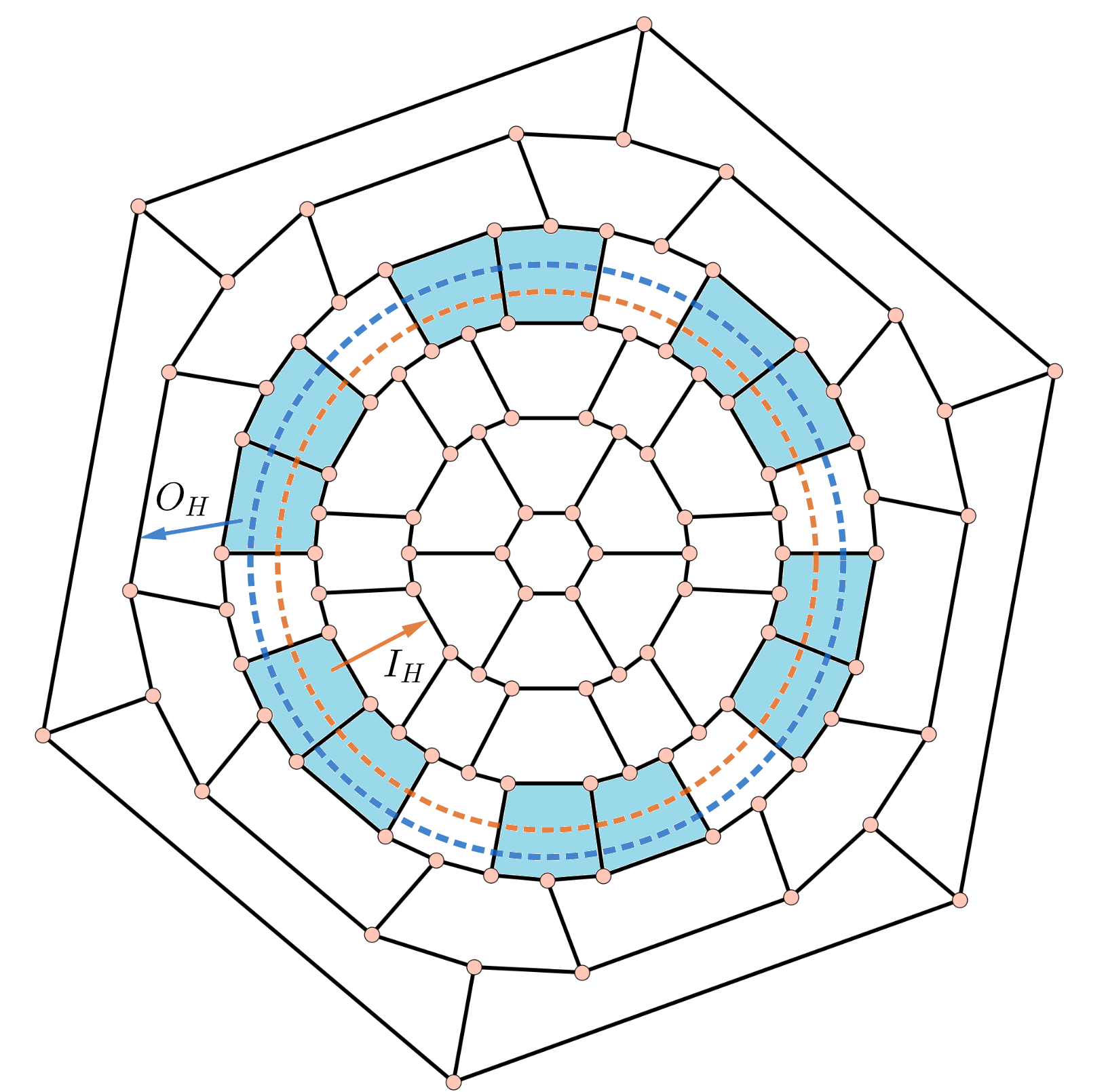}  
   \caption{The inner hemisphere $I_{H}(D_3)$ bounded by the border dotted in orange; The outer hemisphere $O_{H}(D_3)$ bounded by the border dotted in blue.}
    \label{f:D3hemisphere}
\end{figure}
\FloatBarrier

For hemisphere $O_{H}$, we label the cycles as follows:

$$
C^1 := C_{12r-6},
C^2 := C_{12r-18},
\ldots ,
C^{r-1} := C_{18},
C^r := C_6.
$$

To define a nanodisc block, it is necessary to label specific vertices, particularly those that form the frontier of the block. We will label the $12$ consecutive pentagons in the central layer $P_{i}, P_{i}', 1 \leq i \leq 6$. Observe that $P_{i}, 1 \leq i \leq 6$ are partitioned in the same way and we can assume that each of them has two vertices in $C^{1} \subset O_{H}(D_r)$. Similary for $P_{i}', 1 \leq i \leq 6$ in relation to $I_{H}(D_r)$. By construction of $D_r$, there are $r-1$ faces between $P_{i}$ and $P_{i+1}$, one of which is $P_{i}'$, a consecutive pentagon of $P_{i}$ and the remaining are balanced hexagons. For convenience, the figures in this section will depict the pentagons $P_{i}$, $P_{i}'$ as next to each other, although as we have seen in Section~\ref{s:fullerene} they are not required to be next to each other.

We ask the reader to follow Figure~\ref{f:ConstructionPaths} for a better understanding of the construction that will be detailed below. Consider a pair of pentagons $P_i,P_{i+1}, i \in \{1,3,5\}$ in clockwise, and their respective vertices $x_{P_i}, y_{P_i}, x_{P_{i+1}},y_{P_{i+1}} \in C^{1}$. Denote by $u_{L}^{1}$ the vertex $w \in C^{1}$ such that $dist(y_{P_{i}},w) = r$ in anticlockwise, and denote by $u_{R}^{1}$ the vertex $z \in C^{1}$ such that $dist(x_{P_{i+1}},z)=r$ in clockwise.

\begin{figure}[!hbt]
    \centering
    \includegraphics[scale=0.9]{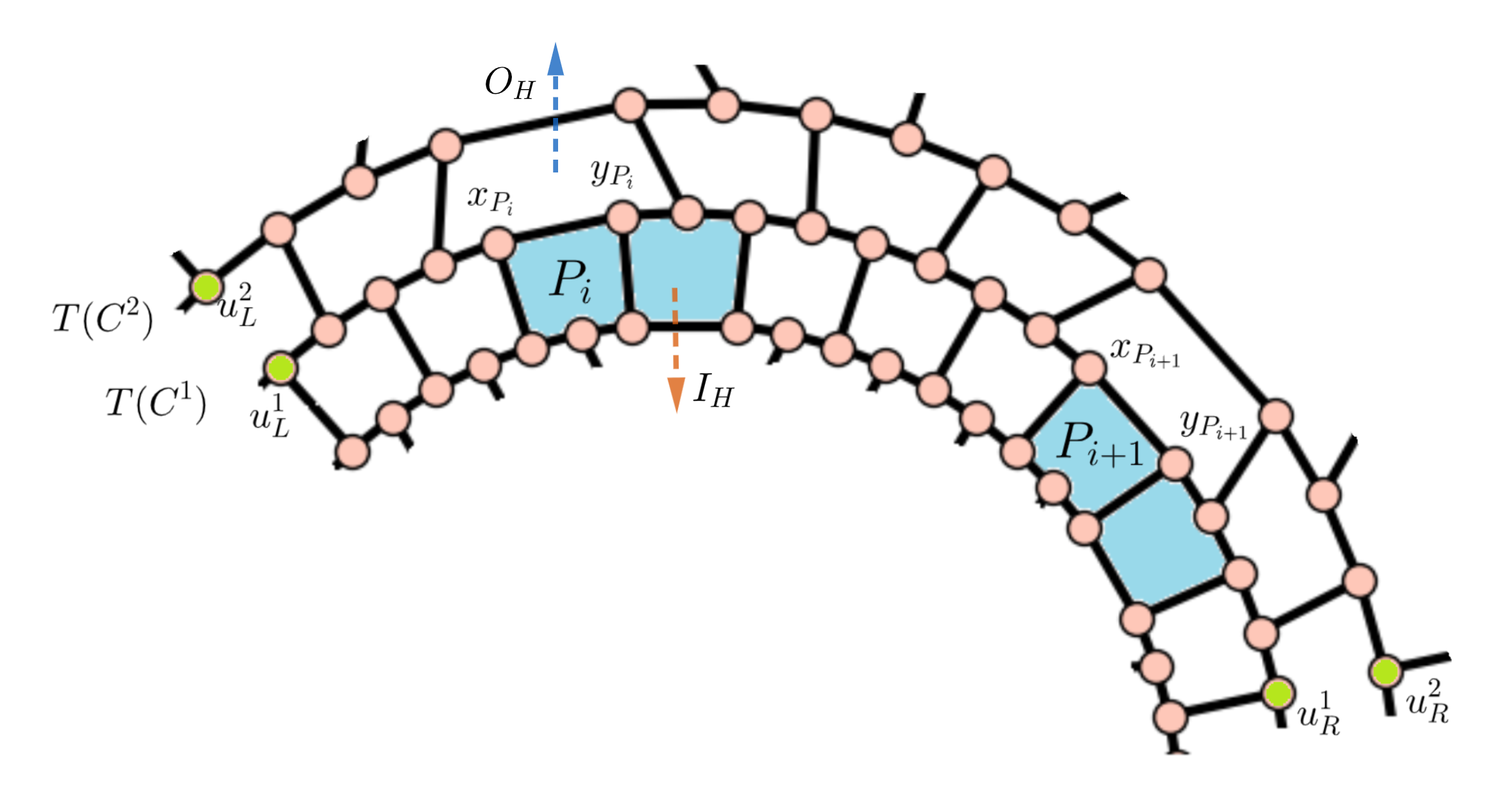}  
   \caption{Construction of $T(C^{1})$ and $T(C^{2})$ paths on the outer hemisphere of fullerene nanodisc $D_5$. The vertices $u_{L}^{1}$, $u_{L}^{2}$, $u_{R}^{1}$, and $u_{R}^{2}$ representing the start and end of the paths are highlighted in lime green. Note that $dist(u_{L}^{1},u_{R}^{1}) = 18 = \frac{1}{3}(|C^{1}|)$ and $dist(u_{L}^{2},u_{R}^{2}) = 14 = \frac{1}{3}(|C^{2}|)$. }
    \label{f:ConstructionPaths}
\end{figure}
\FloatBarrier

Consider the path $T(C^{1})$ in $C^{1}$ from $u_{L}^{1}$ to $u_{R}^{1}$. The following result provides us with the exact length of $T(C^{1})$ in a nanodisc.

\begin{lemma}
\label{l:distancevRvL}
Let $D_r$, $r \geq 2$, be a fullerene nanodisc and the subgraph $T(C^{1})$ as constructed above. The distance between $u_{L}^{1}$ and $u_{R}^{1}$, $d(u_{L}^{1},u_{R}^{1})$ is $\frac{1}{3} \vert C_{12r-6}\vert$.
\end{lemma}
\vspace{-0.1cm}
\begin{proof}

By construction of $T(C^{1})$, $dist(u_{L}^1,u_{R}^1) = dist(u_{L}^1,y_{P_{i}}) + dist(y_{P_{i}}, x_{P_{i+1}}) +  dist(x_{P_{2}},u_{R}^1) = r + 2(r-1) + r = 4r - 2 = \frac{1}{3}({12r-6})$.
\end{proof}

We can construct the path $T(C^{2})$ from $T(C^{1})$ as follows. Choose $u_{L}^{2},u_{R}^{2} \in C^{2}$ such that $dist(u_{L}^{2},u_{R}^{2})= 4r - 6 =\frac{1}{3} |C^{2}|$ and such that the path from $u_{L}^{2}$ to $u_{R}^{2}$ contains all the vertices in $C^{2}$ that have radial edges with extremes in $T(C^{1})$. Figure~\ref{f:ConstructionPaths} exhibits the construction of $T(C^{2})$ from $T(C^{1})$. More generally, we can construct the paths $T(C^{t})$, $t \in \{2,3,\ldots ,r\}$ from $T(C^{t-1})$ constructed previously. We choose $u_{L}^{t},u_{R}^{t}  \in C^{t}$ such that $dist(u_{L}^{t},u_{R}^{t})=\frac{1}{3}|C^{t}|$ and such that the path from $u_{L}^{t}$ to $u_{R}^{t}$ contains all the vertices in $C^{t}$ that have radial edges with extremes in $T(C^{t-1})$.

A \emph{nanodisc block}, denoted by $B_{i,i+1}(D_r)\subset O_{H}(D_r)$, or simply $B_{i,i+1}$, $i \in \{1,3,5\}$, 
is the subcubic semigraph induced by vertices of $T(C^{t}), \forall t \in \{1,2,\ldots ,r\}$. The semiedges of $B_{i,i+1}(D_r)$  are given by $u{\cdot} \in C^{1}$. 
Since the hemispheres are isomorphic, the construction described above is analogous for $I_H(D_r)$ and such block is denoted by $B{'}_{i,i+1}(D_r)$.
Figure~\ref{f:BD5block} displays examples of nanodisc blocks. Note that the arrangement of radial edges between the vertices $u_{L}^{t}, u_{R}^{t}, t \in \{1,\ldots ,r\}$ depends on the parity of the radius $r$, as well as the existence of semi-edges $u_{L}^{1}{\cdot}, u_{R}^{1}{\cdot}$. 

\begin{figure}[!hbt]
    \centering
    \includegraphics[scale=1.07]{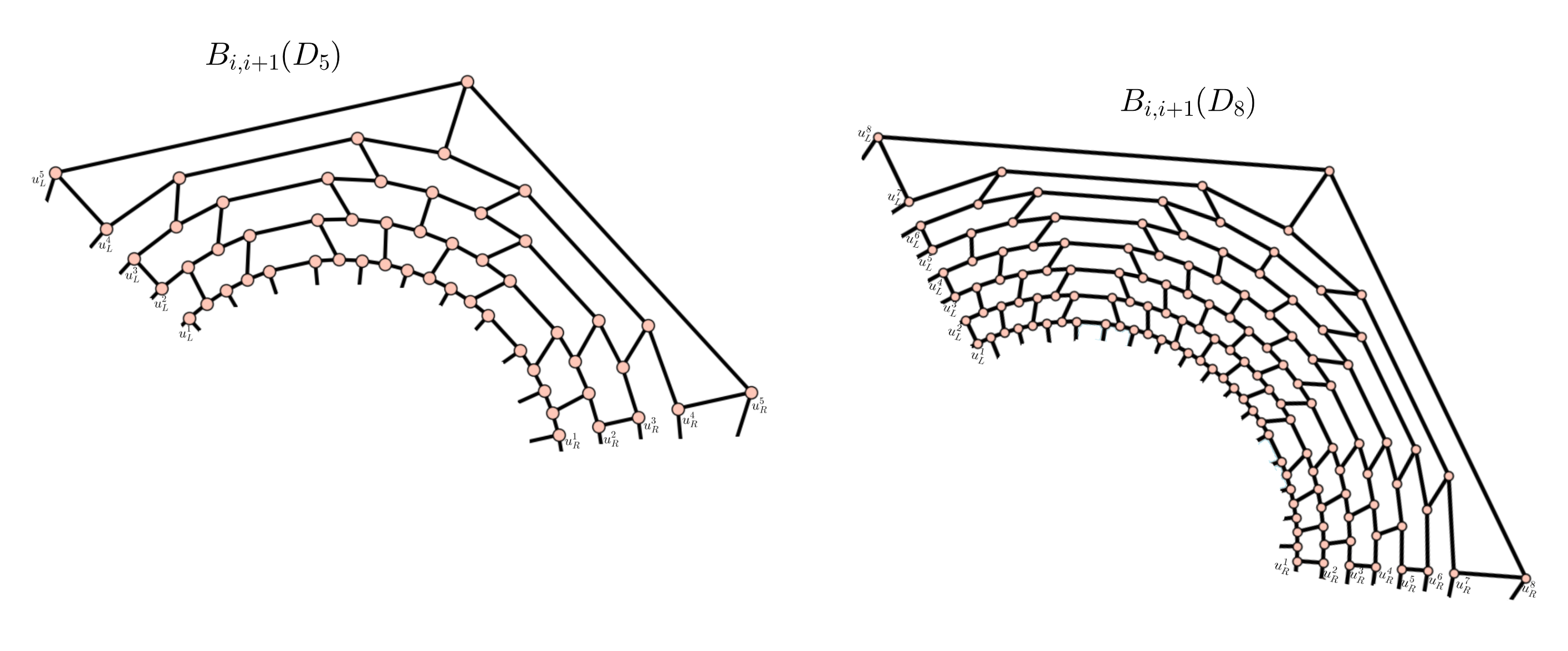}  
   \caption{On the left, a block $B_{i,i+1}$ in outer hemisphere of $D_5$; On the right, a block $B_{i,i+1}$ in outer hemisphere of $D_8$.}
    \label{f:BD5block}
\end{figure}
\FloatBarrier
As each subgraph $T(C^{t}) \in B_{i,i+1}(D_r)$ has $\frac{1}{3}|C^{t}|$ vertices, it is easy to see that $B_{i,i+1}(D_r)$ consists of one third of $O_{H}(D_r)$, as well as $B{'}_{i,i+1}(D_r)$ in relation to $I_{H}(D_r)$. We denote the set of vertices in the left frontier of $B_{i,i+1}(D_r)$ by 
$
V_{L} = \{u_{L}^{1},u_{L}^{2},\ldots ,u_{L}^{t-1},u_{L}^{t}\},
$
the set of edges in the left frontier by
$E_{L}= \{u_{L}^{2}u_{L}^{3}, u_{L}^{4}u_{L}^{5}\ldots u_{L}^{t-1}u_{L}^{t}\}$, if $r$ is odd; $E_{L}= \{u_{L}^{1}u_{L}^{2}, u_{L}^{3}u_{L}^{4}\ldots u_{L}^{t-1}u_{L}^{t}\}$, if $r$ is even, and the set of semiedges in the left frontier by $S_L = \emptyset$ if $r$ even, and $S_L = \{u_{L}^{1}{\cdot}\}$ if $r$ odd. Analogously, we define the set of vertices, edges, and semiedges of the elements in the right frontier, denoted by $V_R, E_R$, and $S_R$, respectively. 

Next, we will define an operation between blocks of the same hemisphere. In this operation, we consider the semigraph induced by vertices of the right frontier of a block, say $B_{1,2}(D_r) \subset O_H(D_r)$ and the semigraph induced by the left frontier of another block, say $B_{3,4}(D_r)$ in the same hemisphere, and we identify the vertices $V_R(B_{1,2}(D_r))$ with the vertices $V_L(B_{3,4}(D_r))$
as well their radial edges and semiedges, if any, and we keep the other elements of the two blocks. Therefore, we can formally define this operation described above as~follows.

\begin{definition}[Junction identifying]
Let $D_r, r\geq 2$ be a fullerene nanodisc and $B_{i,i+1},B_{s,s+1}, i,s \in \{1,3,5\}, i \neq s$, be two blocks of $O_{H} \subset D_r$. The \emph{junction identifying} of $B_{i,i+1}$ and $B_{s,s+1}$, denoted by $B_{i,i+1}\oplus B_{s,s+1}$
is an operation that produces the semigraph such that
$$
V(B_{i,i+1}\oplus B_{s,s+1})= V^{\oplus} \cup V(B_{s,s+1}) \setminus V_{R}(B_{i,i+1} )\cup V(B_{s,s+1}\setminus V_{L}(B_{s,s+1}), 
$$
$$
E(B_{i,i+1}\oplus B_{s,s+1}) = E^{\oplus} \cup E(B_{i,i+1}) \setminus E_{R}(B_{i,i+1}) \cup E(B_{s,s+1}) \setminus E_L(B_{s,s+1}),
$$
$$
S(B_{i,i+1}\oplus B_{s,s+1}) = S^{\oplus} \cup S(B_{i,i+1})\setminus S_{R}(B_{i,i+1}) \cup S(B_{s,s+1})\setminus S_{L}(B_{s,s+1}),
$$
\end{definition}
where $V^{\oplus},E^{\oplus},S^{\oplus}$ are the sets of vertices, edges, and semi-edges resulting from the identification of the elements $V_{R}(B_{i,i+1}),E_{R}(B_{i,i+1}),S_{R}(B_{i,i+1})$ with $V_{L}(B_{s,s+1}), E_{L}(B_{s,s+1}), S_{L}(B_{s,s+1})$, respectively. By the symmetry of the graph, the junction identifying operation is analogous for blocks $B{'}_{i,i+1}(D_r) \subset I_{H}$. Figure~\ref{f:JIDD5} illustrates the resulting semigraph of this operation in outer hemisphere.

\begin{figure}[!hbt]
    \centering
    \includegraphics[scale=1.60]{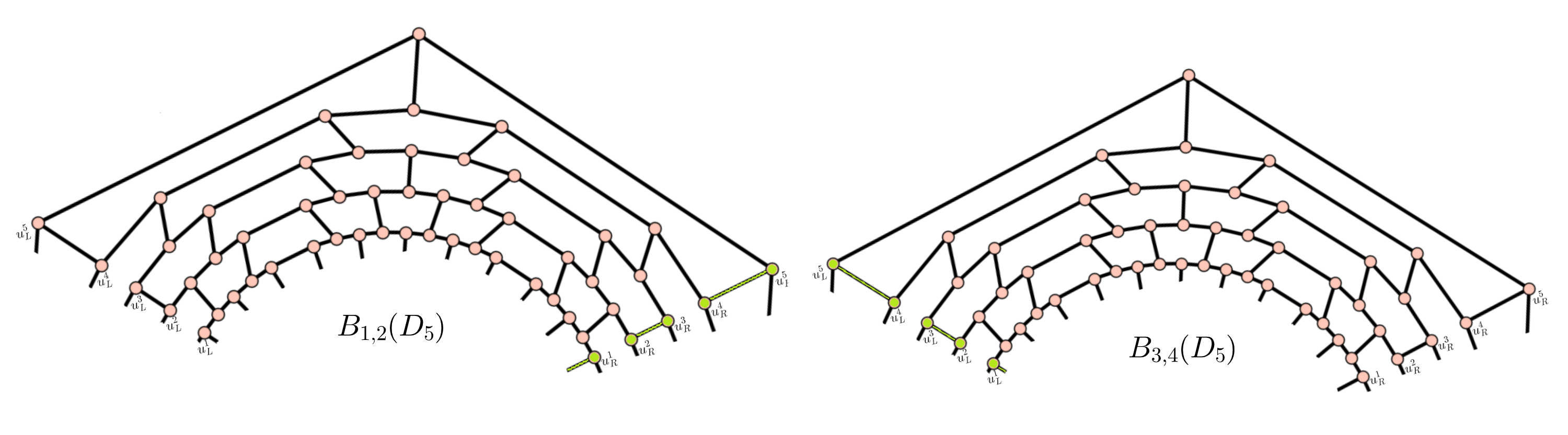}
    \centering
     \includegraphics[scale=2.09]{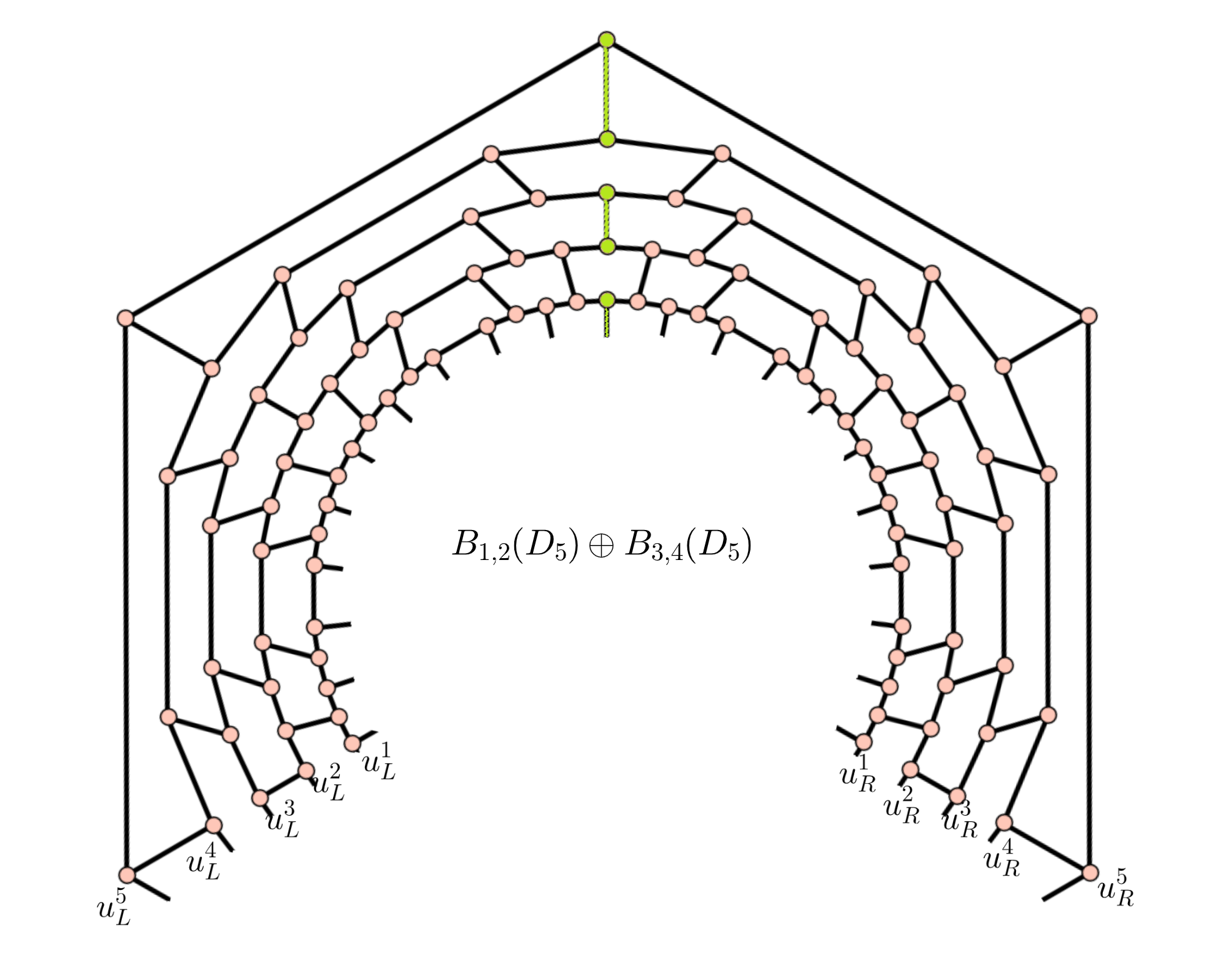}
   \caption{The junction identifying operation between two blocks in outer hemisphere of $D_5$.}
    \label{f:JIDD5}
\end{figure}

Let $B_{1,2},B_{3,4},B_{5,6}$ be three blocks of outer hemisphere of a nanodisc $D_r$, $r \geq 2$. Clearly, if we apply the operations (($B_{1,2} \oplus B_{3,4}) \oplus B_{5,6}) \oplus B_{1,2})$ the resulting semigraph is the outer hemisphere $O_{H}(D_r)$. Analogously, if we operate three blocks of the inner hemisphere in the same sequence, the semigraph obtained is the inner hemisphere $I_{H}(D_r)$.

We can also define an operation between blocks from different hemispheres. This operation consists of joining blocks in an appropriate way through their semiedges, taking advantage of the symmetry of the nanodisc by reflection. 

\begin{definition}[Junction operation]
Let $D_{r}, r\geq 2$ be a fullerene nanodisc and $B_{i,i+1}$, $B{'}_{i,i+1}$ be two blocks of $D_r$ of outer and inner hemispheres, respectively. The \emph{junction} between $B_{i,i+1}, B{'}_{i,i+1}$, denoted by $B_{i,i+1}\cup B{'}_{i,i+1}$ is an operation that produces the semigraph such that
$$
V(B_{i,i+1}\cup B{'}_{i,i+1}) = V(B(D_r)) \cup V(B'(D_r));
$$
$$
E(B_{i,i+1}\cup B{'}_{i,i+1}) = E(B_{i,i+1}) \cup E(B{'}_{i,i+1})\cup E''),
$$
where $E''$ is obtained by junction of semiedges such that we have two cases to consider: 
\begin{enumerate}
    \item If the pentagons $P_{i},P_{i}'$ are next to each other, we join the second semiedge of $B_{i,i+1}$ in clockwise with the first semiedge of $B{'}_{i,i+1}$ in clockwise.  The remaining $2r-1$ semiedges of $B_{i,i+1}$ are joined with $2r-1$ semiedges of $B{'}_{i,i+1}$ in a subsequent clockwise manner. Note that in this construction, one semiedge of each block is not joined to any semiedge. Thus, the set of semiedges $S(B_{i,i+1}\cup B{'}_{i,i+1})$ contains two semiedges, that is, $S^{(B_{i,i+1}\cup B{'}_{i,i+1})} = \{u{\cdot},v{\cdot}\}, u \in B_{i,i+1}$ and $v \in B{'}_{i,i+1}$.
    \item If the pentagons $P_{i},P_{i}'$ are consecutive, we join the $(\lfloor\frac{r+1}{3}\rfloor)$-th semiedge of $B_{i,i+1}$ in clockwise with the first semiedge of $B{'}_{i,i+1}$ in clockwise. The remaining  semiedges of $B_{i,i+1}$ are joined with semiedges of $B{'}_{i,i+1}$ in a subsequent clockwise manner. Note that in this construction, if $(\lfloor\frac{r+1}{3}\rfloor)$ is an even integer, $2(\lfloor\frac{r+1}{3}\rfloor)$ semiedges were not joined, being half of each block; if $(\lfloor\frac{r+1}{3}\rfloor)$ is an odd integer, $2(\lfloor\frac{r+1}{3}\rfloor-1)$ semiedges were not joined, being half of each block.

\end{enumerate}

\end{definition}



Note that any fullerene nanodisc $D_{r}$, $r\geq 2$  can be obtained by junction of the semigraphs $I_{H}(D_r)$ and $O_{H}(D_r)$. Figure~\ref{f:JunctionOP} illustrates the junction between blocks of different hemispheres.

\begin{figure}[!hbt]
    \centering
    \includegraphics[scale=0.48]{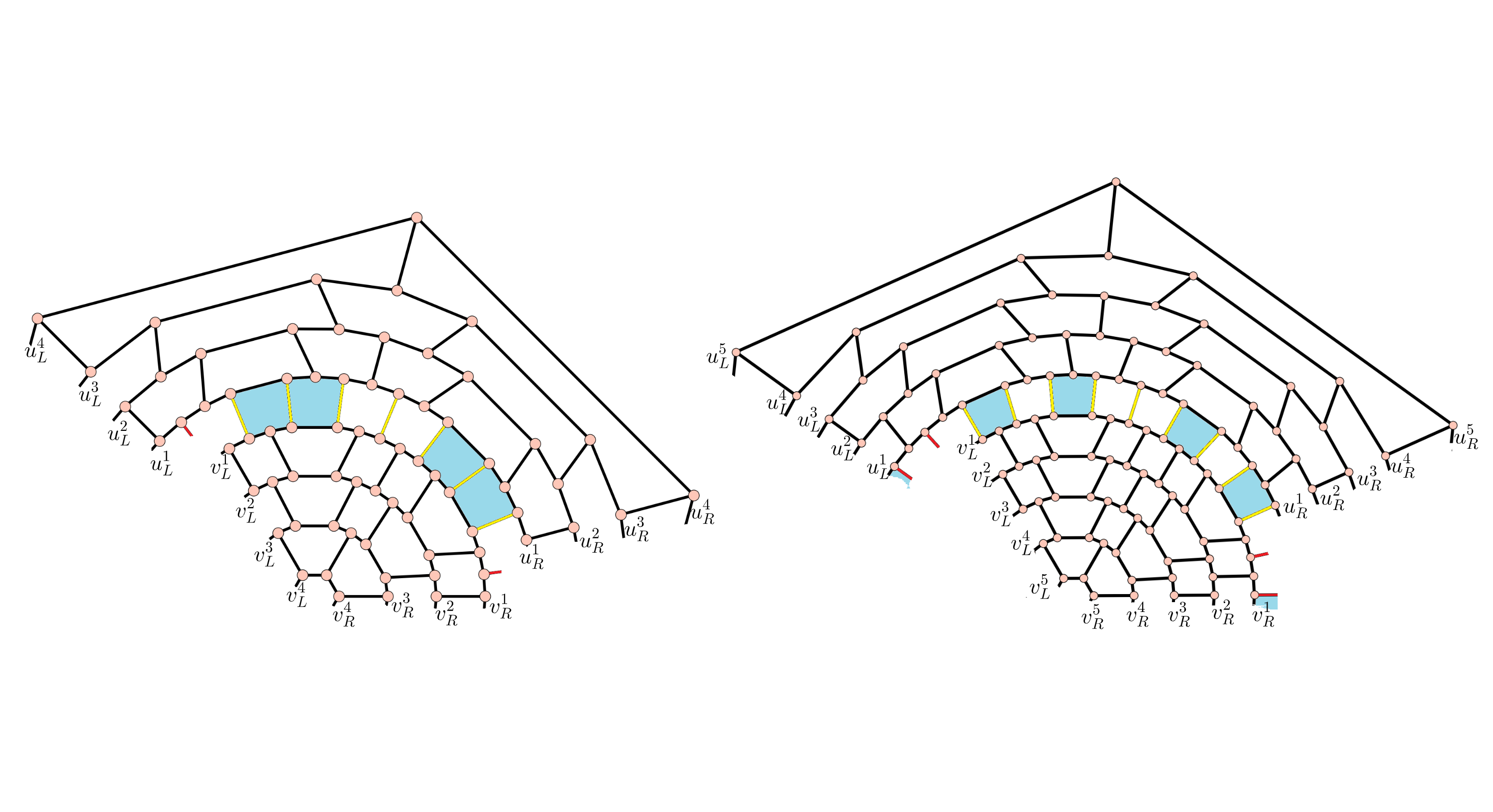}
    \caption{Junction between two blocks of different hemispheres of $D_4$ and $D_5$, respectively. The radial edges highlighted in yellow are the edges of $E''$ obtained by junction of semiedges, and the remaining semiedges are highlighted in red.}
    \label{f:JunctionOP}
\end{figure}
\FloatBarrier

This suitable decomposition will serve as a tool to obtain the main goal of this paper, as we will see in the next section.





\section{The fullerene nanodiscs $D_{r}, r = 5+3k$ are Type 1}
\label{s:Type1}

In this section, we prove that the fullerene nanodiscs 
$D_{r}, r = 5+3k$
have total chromatic number 4. Lemma~\ref{l:centrallayer4totalc} provides an important tool for the final result of this paper. The $4$-total coloring obtained for the target infinite family will be obtained by extending the $4$-total coloring given for central layer of $D_r$.
At first we restrict our study to a specific subfamily of graphs in this subclass, namely the nanodiscs $D_r$, $r = 5+3k$ whose pentagonal representation is in pairs, that is, the pentagons $P_{i}$, $P_i'$, $1\leq i \leq 6$, are next to each other. Such a representation exists for every nanodisc $D_r$, as we have seen in Section~\ref{s:fullerene}.



\begin{figure}[!hbt]
    \centering
    \includegraphics[scale=0.70]{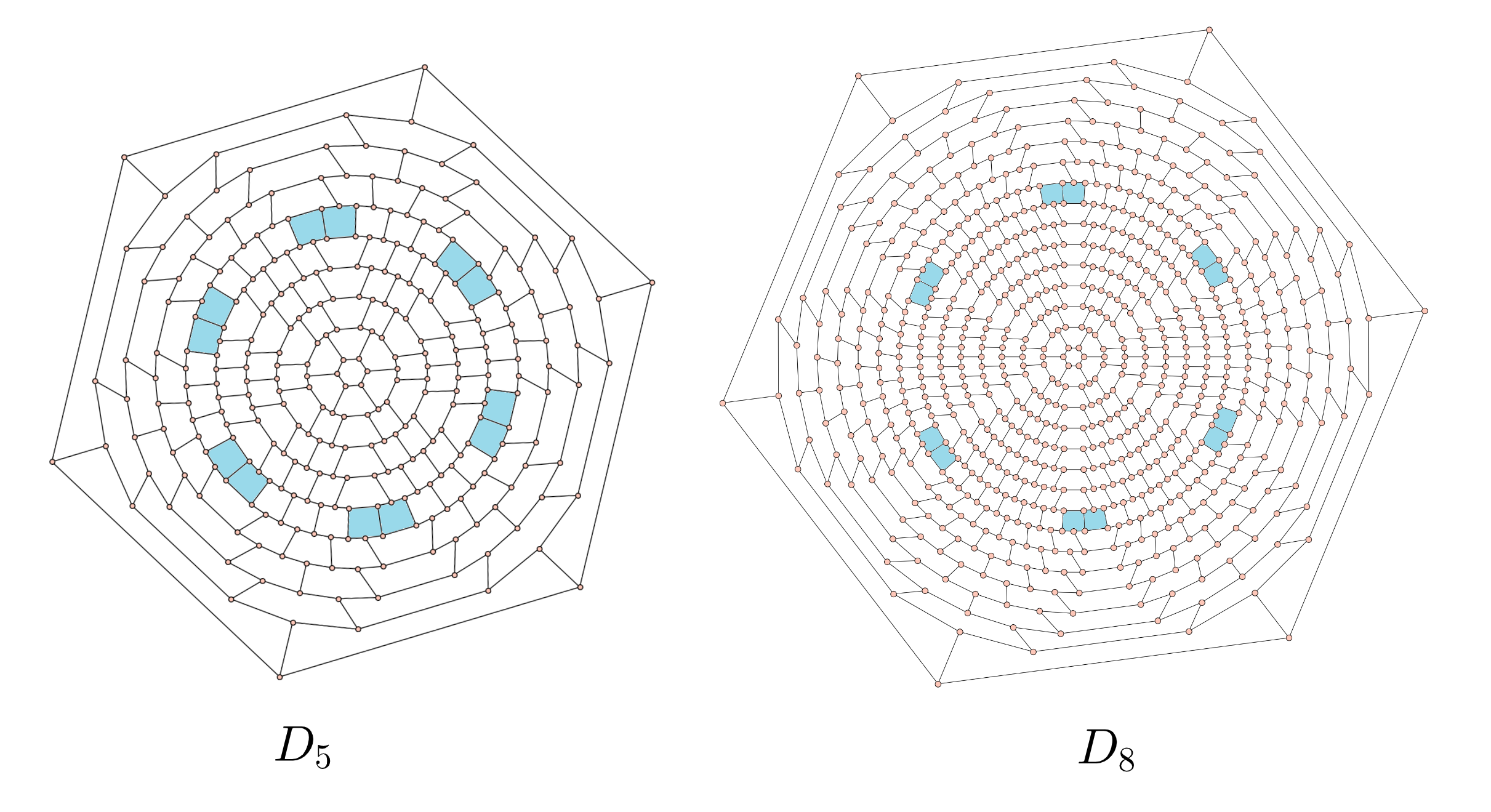}  
   \caption{For an arbitrary radius, there exist a nanodisc where  the pentagons occur in pairs and the hexagons occur in groups of  size a multiple of 3.}
    \label{f:D5andD8}
\end{figure}
\FloatBarrier

These graphs have an important property in the central layer, that is suitable to define a 4-total coloring with a circular symmetry.

\begin{fact}\label{fact1}
    The fullerene nanodiscs with $D_r$, $r = 5+3k$ have $3(k+1)$ hexagons between each pair of pentagons $P_{i},P_{i+1}$, $i \in \{1,3,5\}$, that is, the number of hexagons between each pair of pentagons partitioned in the same way is multiple of 3.
\end{fact}

See examples in Figure~\ref{f:D5andD8}. By Lemma~\ref{l:centrallayer4totalc}, the $3$-total coloring attributed to these graphs has a suitable coloring structure as a result of the number of balanced hexagons being multiple of 3: every pentagon $P_i$ has the same coloring structure and every pentagon $P_i'$ also has the same coloring structure, for every $i \in \{1,2, \ldots 6\}$. Figure~\ref{f:D5wYapproperty} illustrates a $4$-total coloring obtained for the central layer of $D_5$ using Lemma~\ref{l:centrallayer4totalc}.

\begin{figure}[!hbt]
    \centering
    \includegraphics[scale=5.33]{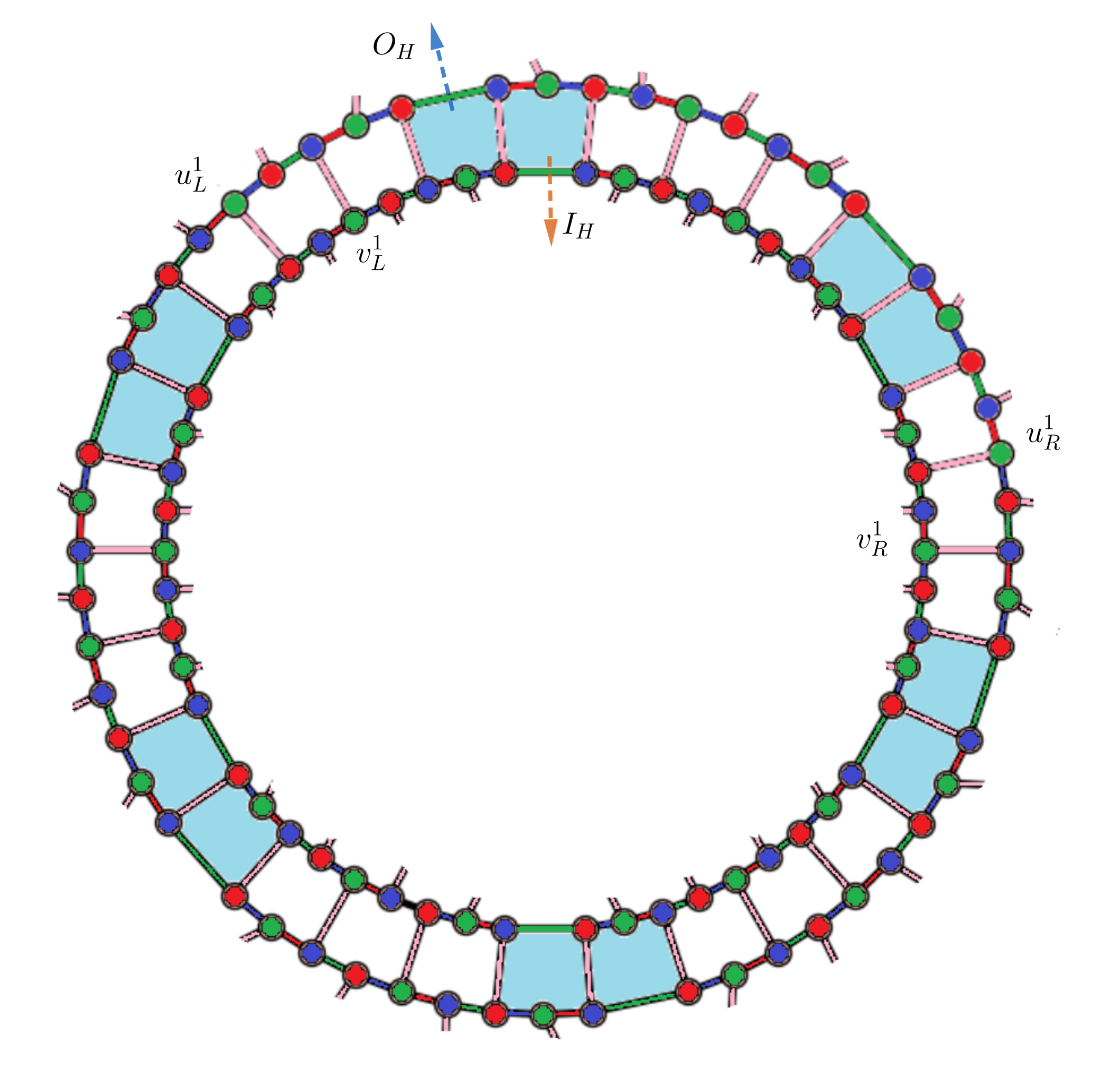}  
   \caption{The 4-total coloring of the central layer of a $D_5$. Three colors are used in each auxiliary cycle $C_{54}$. In any extension to a 4-total coloring of the whole graph, all radial edges incident to the central layer must be colored with color 4. We denote the color red as $1$, green as $2$, blue as $3$, and pink as $4$.
   Observe the two labeled  vertices $v_{L}^{1}$ and $v_{R}^{1}$ of the central layer. Both $v_{L}^{1}$ and $v_{R}^{1}$ receive color $2$. Vertices $v_{L}^{1}$, $v_{R}^{1}$, and the 17 vertices between them are block vertices. }
    \label{f:D5wYapproperty}
\end{figure}
\FloatBarrier



Consider the nanodisc block $B(D_r) \subset O_H(D_r)$. As an auxiliary result, we will prove that such semigraph is Type 1, and by isomorphism, $B'(D_r) \subset I_H(D_r)$ is also Type 1. Finally, the inter-block operations defined in Section~\ref{s:block} show that the coloring of a block defines a $4$-total coloring for every $D_r$, $r = 5+3k$, $k \in \mathbb{Z}^{+}$, with the properties defined in this section.




\begin{figure}[!hbt]
    \centering
    \includegraphics[scale=0.83]{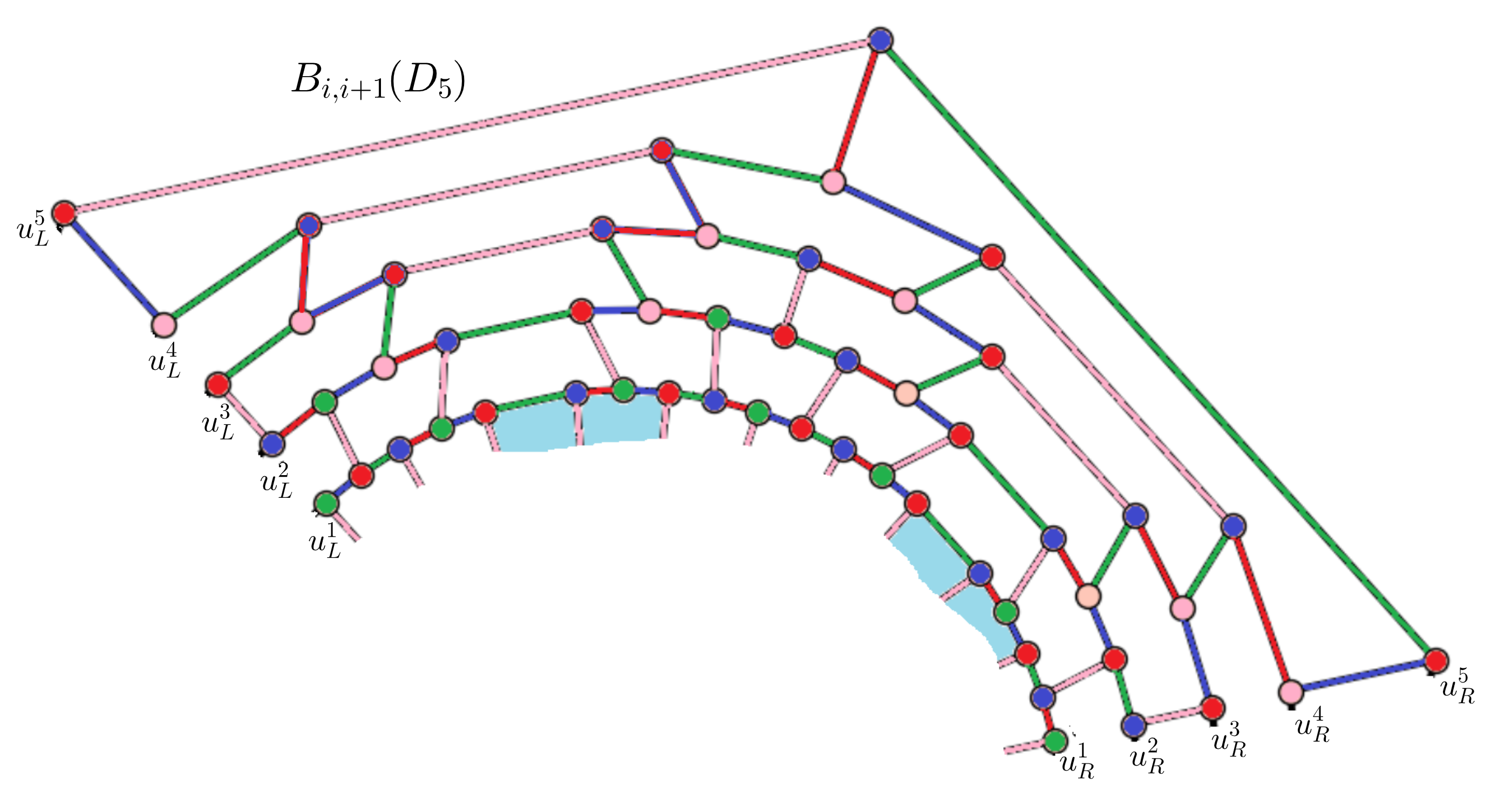}  
   \caption{A block of $D_5$ in outer hemisphere colored with a $4$-total coloring, satisfying that all radial edges incident to the central layer are colored with color 4.}
    \label{f:blockD54total}
\end{figure}
\FloatBarrier

\begin{figure}[!hbt]
    \centering
    \includegraphics[scale=0.55]{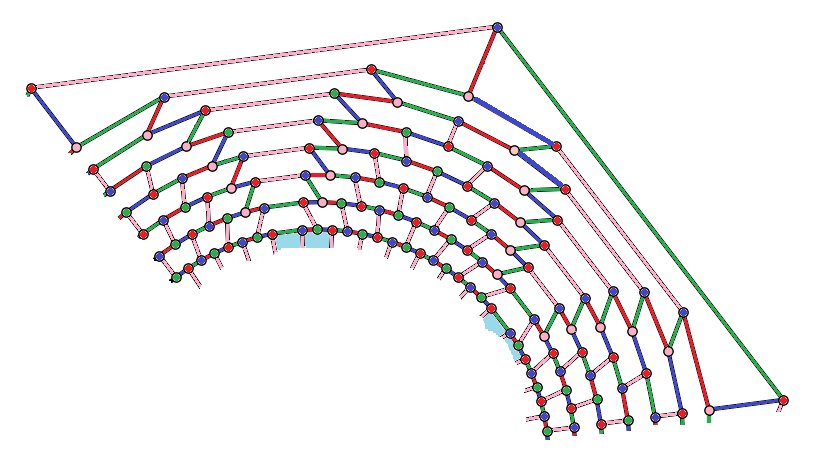}  
   \caption{The $4$-total coloring of a block of $D_{8}$ in outer hemisphere provided by Lemma~\ref{l:blocks4colorable}.}
    \label{f:BD84totalcoloring}
\end{figure}
\FloatBarrier

\begin{figure}[!hbt]
    \centering
    \includegraphics[scale=0.30]{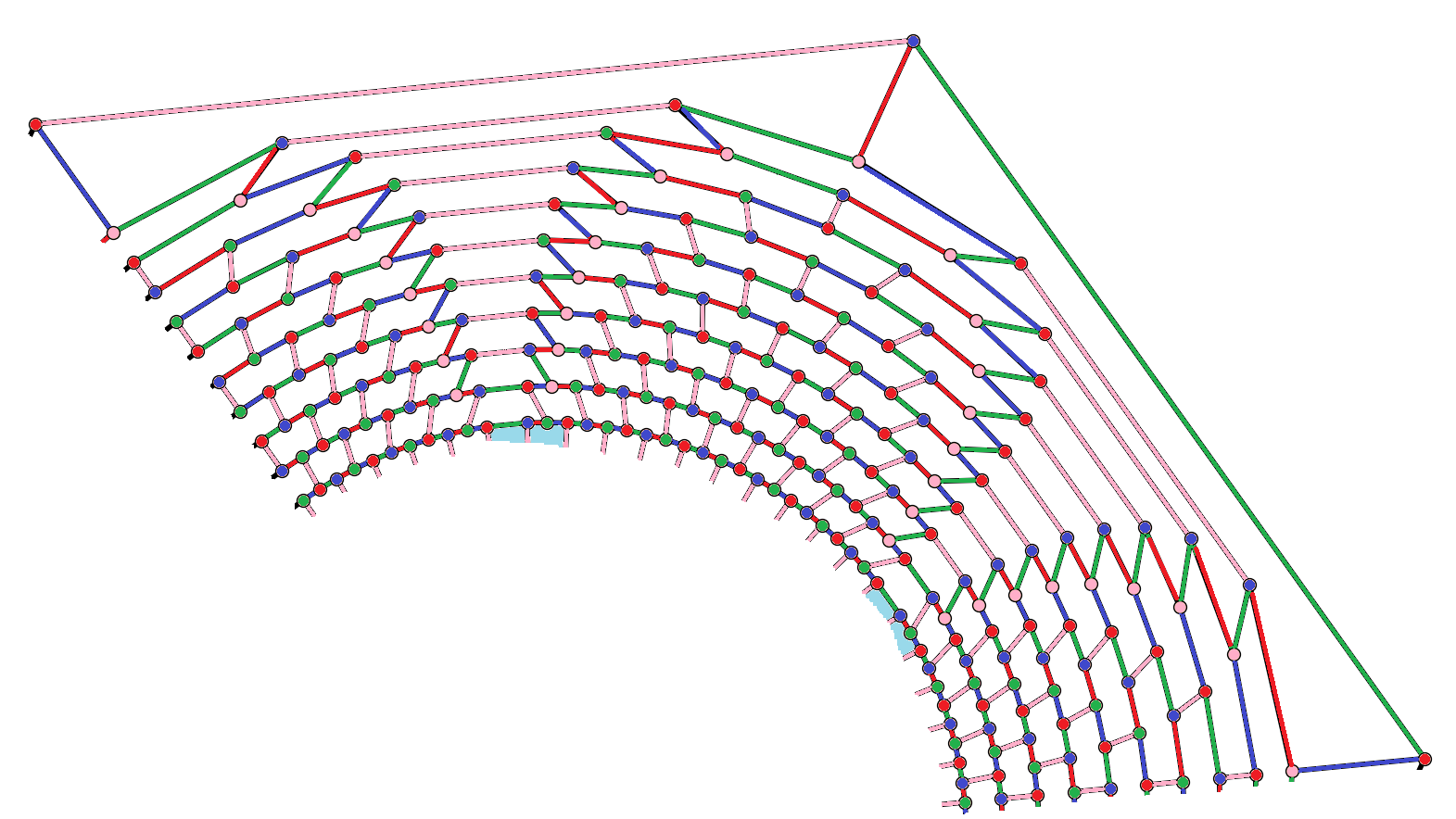}  
   \caption{The $4$-total coloring of a block of $D_{11}$ in outer hemisphere provided by Lemma~\ref{l:blocks4colorable}.}
    \label{f:BD11}
\end{figure}
\FloatBarrier




\begin{lemma}\label{l:blocks4colorable}
The blocks $B_{i,i+1}(D_r) \subset O_H$, 
$r = 5+3k, k \in \mathbb{Z}^{+}$,
are $4$-total colorable.

\end{lemma}
\begin{proof}
We ask the reader to follow Figures~\ref{f:blockD54total}, \ref{f:BD84totalcoloring} and~\ref{f:BD11} to understand the coloring structure defined below. 
Let $D_r$, $r = 5+3k$, $k \in \mathbb{Z}^{+}$ be a fullerene nanodisc and $B_{i,i+1} \subset O_{H}$ be a block of $D_r$. The assignment of colors to the elements of $T(C^{1})$ follows from Lemma~\ref{l:centrallayer4totalc}. Note that by Lemma~\ref{l:distancevRvL}, $dist(u_{L}^{1},u_{R}^{1}) = 4r-2$, and since $r = 5+3k$, the distance between $u_{L}^{1},u_{R} ^{1}$ is a multiple of 3, and therefore $u_{L}^{1}$and $u_{R}^{1}$ receive the same color. 
In this case, we will assume that $c(u_{L}^{1})=c(u_{R}^{1}) = 2$.
We will define the total coloring of each path $T(C^{t})$ with the set of colors $\{1,2,3,4\}$. By counting of vertices, $T(C^{t})$ has $4(r-t)+3$ vertices, $1 \leq t \leq r$. Observe that each layer between $T(C^{t})$ and $T(C^{t+1})$, $1 \leq t \leq r-1$, contains two unbalanced hexagons within $B_{i,i+1}(D_r)$. Each of these hexagons contains 4 vertices in $T(C^{t})$ and 2 vertices in $T(C^{t+1})$.
Given these two unbalanced hexagon, every path may be divided into 5 smaller paths: a path of size $(r-t-1)$ starting at $u_{R}^{t}$ to the first vertex of $T(C^{t})$ that belongs to the first unbalanced hexagon;
a path consisting of the 4 vertices of the first unbalanced hexagon;
a path of size $2(r-t)-3$ starting at the last vertex of $T(C^{t})$ that belongs to the first unbalanced hexagon to the first vertex of $T(C^{t})$ that belongs to the second unbalanced hexagon;
the path that contains the 4 vertices of the second unbalanced hexagon and the path of size $(r-t-1)$ starting at the last vertex of $T(C^{t})$ that belongs to the second unbalanced hexagon and ending at $u_{L}^{t}$. 

Up to define the desired total coloring based on the above structure of each path, we will divide all $T(C^{t})$, $1 \leq t \leq r$, in two cases: the base case and the inductive step.
The base case consists of the 5 paths $T(C^{1}),T(C^{2}), T(C^{3}), T(C^{r-1})$ and $T(C^{r})$, see Figure~\ref{f:tctpathsbasis}. Note that for $r=5$, every $T({C^{t}})$ is a path in the base case for $B_{i,i+1}(D_5)$. We define the sequence of colors to every base case according to Figure~\ref{f:tctpathsbasis}.

\begin{figure}[!hbt]
    \centering
    \includegraphics[scale=0.98]{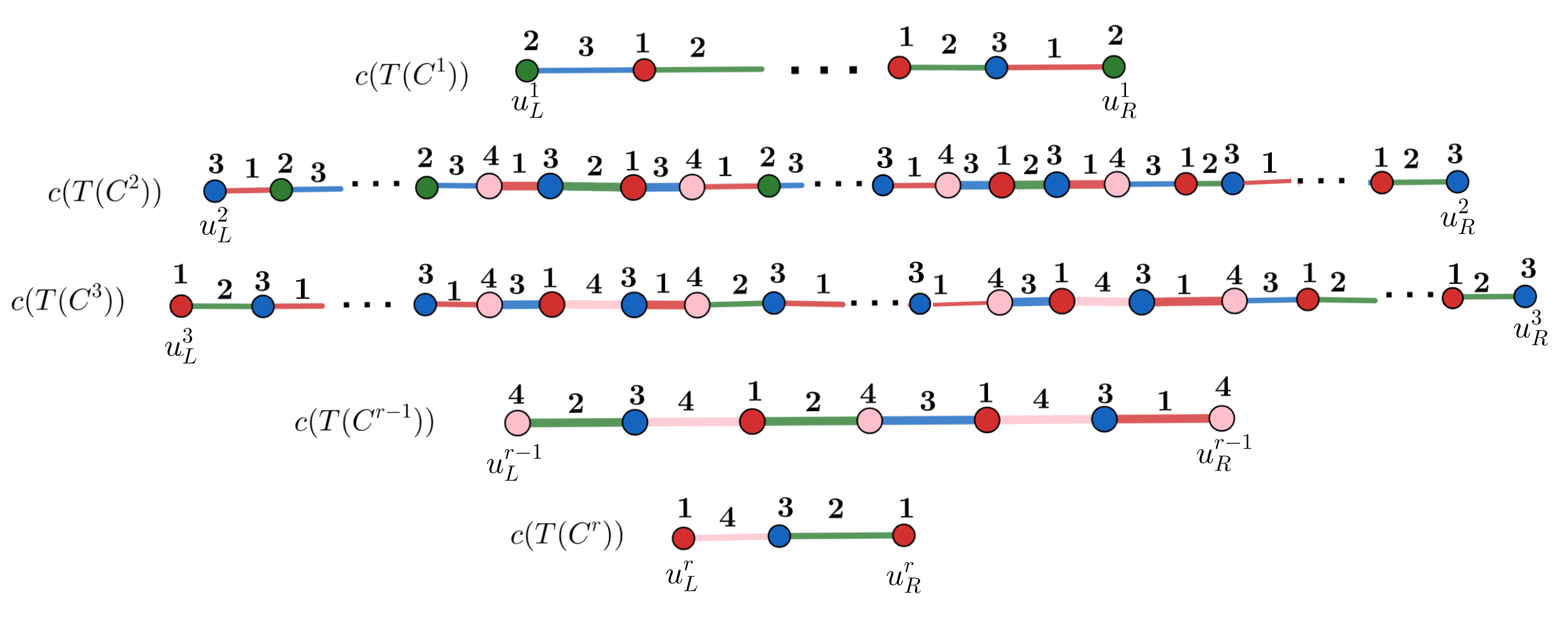}  \caption{Sequence of colors of the base case.}
    \label{f:tctpathsbasis}
\end{figure}
\FloatBarrier

The inductive step consists of the remaining paths $T(C^{t}), 4 \leq t \leq r-2, r = 5+3k, k > 1$, based on the repetition below and the total coloring of $T(C^{4}),T(C^{5}), T(C^{6})$, see Figure~\ref{f:repetitionpaths}.
$$
c(T(C^{t})) = \begin{cases} c (T(C^{4})), &\mbox{if } t \equiv 1\mod 3, 4 \leq t \leq r-2;\\
c(T(C^{5})), &\mbox{if } t \equiv 2\mod 3, 5 \leq t \leq r-2;\\
c(T(C^{6})), &\mbox{if } t \equiv 0\mod 3, 6 \leq t \leq r-2.
\end{cases}
$$

\begin{figure}[!hbt]
    \centering
    \includegraphics[scale=0.98]{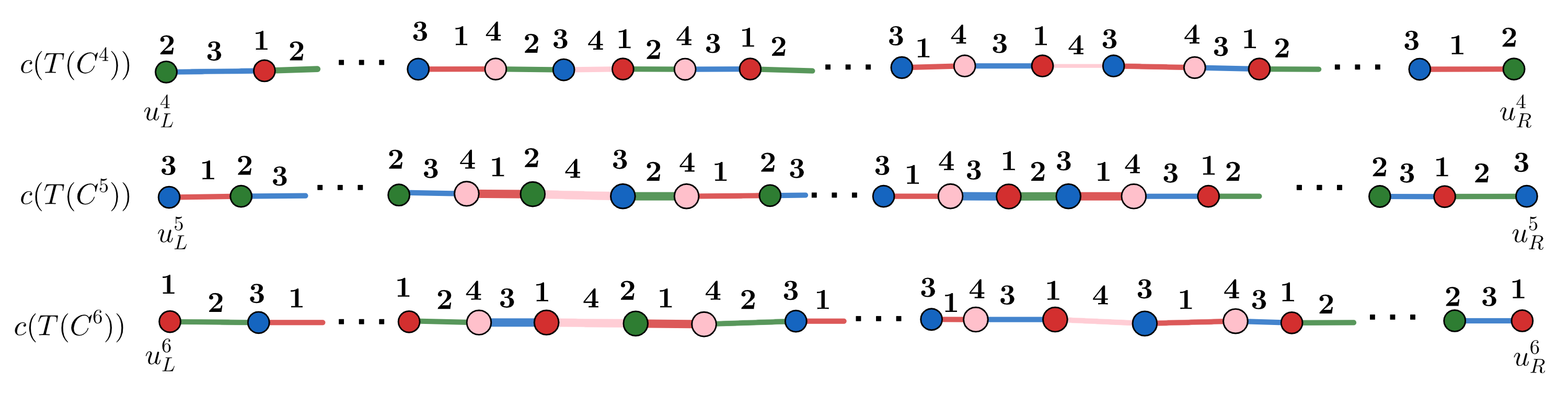}  \caption{Sequence of colors of repetition paths $T(C^{4}), T(C^{5})$ and $T(C^{6})$.}
    \label{f:repetitionpaths}
\end{figure}
\FloatBarrier

For the radial edges that are edges of the unbalanced hexagons, it is straightforward to check that in each case there are no conflicts between the vertices and a single color remains for the radial edge. Recall that the vertices of $T(C^{t})$ that do not belong to an unbalanced hexagon define 3 paths, and each path is colored according to the order 2 (green), 3 (blue), 1 (red), like it was done for $T(C^{1})$, see pattern in Figures~\ref{f:tctpathsbasis} and~\ref{f:repetitionpaths}. For each of these paths, one can easily verify that there is no conflict for the first radial edge that can be colored with color 4 (pink).
Moreover, for the other radial edges, as they are edges of balanced hexagons and join two vertices at the same distance of the first radial edge, they will connect two vertices that do not have the same color and can be colored with color 4 (pink).
Since the elements of $T(C^{1})$ are colored with 3 colors, assign the color 4 (pink) to every semiedge of $B_{i,i+1}(D_r)$ and therefore it is concluded that $B_{i,i+1}(D_r)$, $r = 5+3k, k \in \mathbb{Z}^{+}$ is 4-total colorable.
Since $B_{i,i+1}$ and $B{'}_{i,i+1}$ are isomorphic semigraphs, the $c$ 
coloring described above is also a 4-total coloring for $B{'}_{i,i+1}$. Thus, it is concluded that such blocks are Type 1. 
\end{proof}

Given two $4$-total colorings $c, c'$
of two blocks $B_{i,i+1}$, $B_{s,s+1}$, with $(i,s) \in \{(1, 3),(3, 5), (5, 1)\}$, of a nanodisc $D_r$, $r \geq 2$, in the same hemisphere. We say that $c$ 
is \emph{compatible} with $c'$ 
if:

\begin{itemize}
    \item The color assigned by $c$ to $u_{R}^{t}$ in $B_{i,i+1}$ is the same as the color assigned by $c'$ to $u_{L}^{t}$ in $B_{s,s+1}$,  for each $t \in \{1,2,\ldots r-1,r\}$;
    \item All $2r+1$ semiedges of each block have the same color;
    \item The color assigned by $c$ to the radial edge incident to $u_{R}^{t}$ in $B_{i,i+1}$ is the same as the color assigned by $c'$ to the radial edge incident to $u_{L}^{t}$ in $B_{s,s+1}$,  for each $t \in \{2,\ldots r-1,r\}$;
    \item The color assigned by $c$ to the edge incident to $u_{R}^{t}$, that is not a radial edge, in $B_{i,i+1}$ is different from the color assigned by $c'$ to the edge incident to $u_{L}^{t}$, that is not a radial edge, in $B_{s,s+1}$,  for each $t \in \{2,\ldots r-1,r\}$.
\end{itemize}

Hence, $c$ and $c'$ 
generate a $4$-total coloring for a junction identifying of $B_{i,i+1}$ and $B_{s,s+1}$, and analogously for blocks of the inner hemisphere. 
The $4$-total colorings for blocks $B_{i,i+1}$, $B{'}_{i,i+1}$, $i \in \{1,3,5\}$, of $D_r$ 
for every $r \geq 2$, are called \emph{mutually compatible} if, for every two colorings 
$c,c'$ 
(not necessarily distinct) of the nanodisc blocks 
$B_{i,i+1} \subset O_H, B{'}_{i,i+1}\subset I_H$:

\begin{itemize}
    \item All $2r+1$ semiedges of each block have the same color;
    \item The color assigned to the extreme vertices of any semiedge in $B_{i,i+1}$ is different from the one assigned to the extreme of the corresponding semiedge in $B{'}_{i,i+1}$.
\end{itemize}

This implies that any sequence of junction identifying  and also of junction between blocks provides a $4$-total coloring for these semigraphs. Figure~\ref{f:compatibleblocks} shows a $4$-total coloring obtained from compatible colorings for the semigraph generated by junction identifying of two blocks of the outer hemisphere of $D_5$.

\begin{figure}[!hbt]
    \centering
    \includegraphics[scale=0.60]{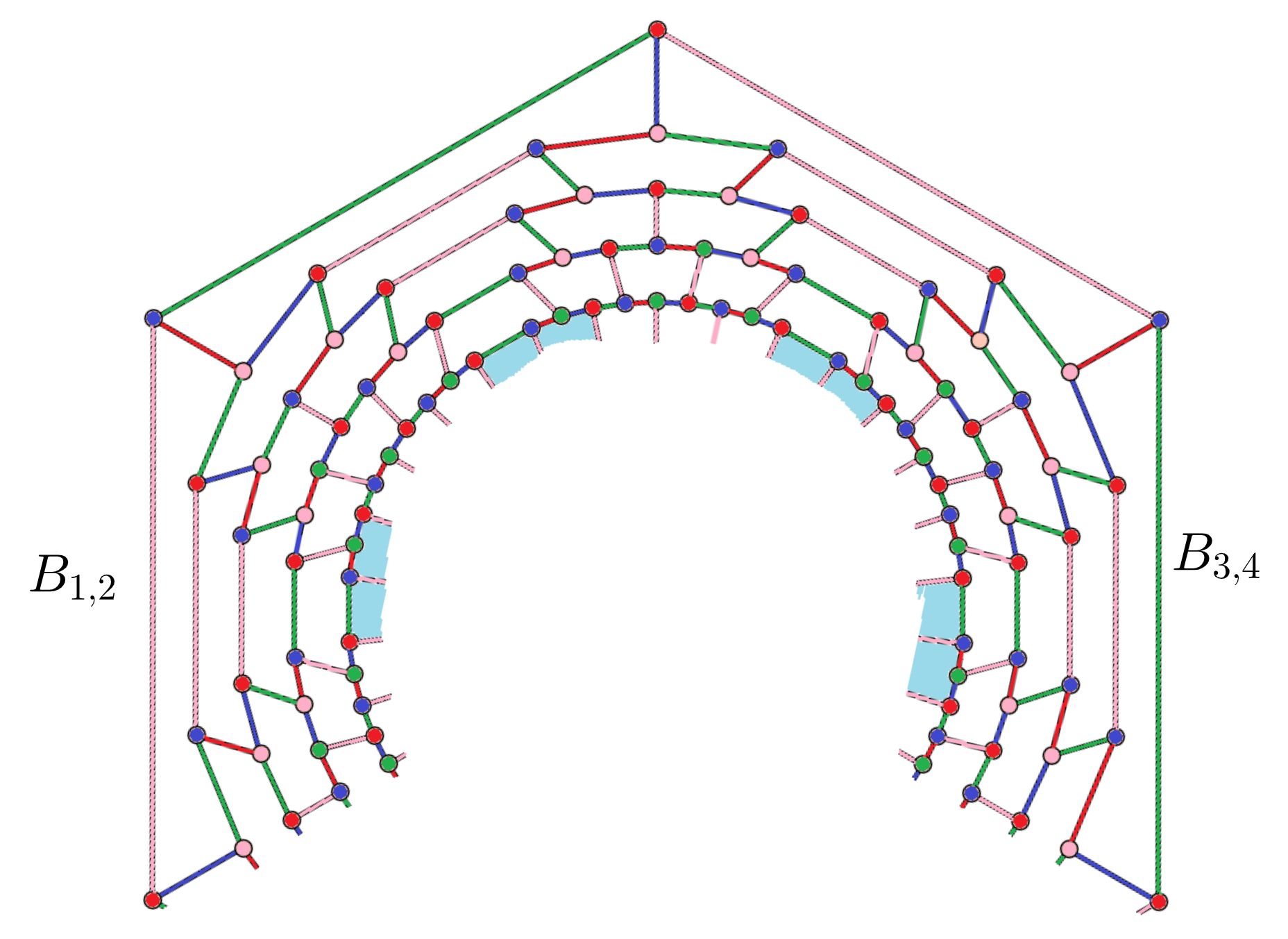}  
    \caption{A $4$-total coloring obtained from compatible colorings for two blocks $B_{1,2}, B_{3,4}$ in the outer hemisphere of $D_{5}$.}
    \label{f:compatibleblocks}
\end{figure}
\FloatBarrier

\begin{theorem}
    The fullerene nanodiscs $D_{r}$, with $r = 5+3k$, $k \in \mathbb{Z}^{+}$, are~Type~1.
\end{theorem}
\begin{proof}
We can obtain any nanodisc $D_r, r \geq 2$ from successive operations of junction identifying between blocks of the same hemisphere and junction between blocks of different hemispheres, as seen in Section~\ref{s:block}. Let $B_{1,2},B_{3,4},B_{5,6}$ be three blocks of $D_r$, $r = 5+3k, k \in \mathbb{Z}^{+}$ in the outer hemisphere. First, given the 4-total coloring $c$ 
shown in Lemma~\ref{l:blocks4colorable} for the three blocks, we must show that
the colorings of $B_{i,i+1}$, $B_{s,s+1}$, $(i,s) \in \{(1, 3),(3, 5), (5, 1)\}$ are compatible.
To prove that the colorings of $B_{1,2}$ and $B_{3,4}$ are compatible, note that every vertex in the frontier $V_{R}(B_{1,2})$ received the same color of $V_{L}(B_{3,4})$, as well the radial edges $E_{R}(B_{1,2})$ and $E_{L}(B_{3,4})$, and the remaining edges incident to each frontier vertex received distinct colors, i.e., there is no color conflict. Therefore, the coloring $c$ given for $B_{1,2}$ and $B_{3,4}$ provides a 4-total coloring for the semigraph generated by $B_{1,2} \oplus B_{3,4}$. The same argument can be given for $B_{3,4} \oplus B_{5,6}$ and $B_{5,6} \oplus B_{1,2}$. As we have seen in Section~\ref{s:block}, by combining such blocks, we obtain a 4-total coloring for the outer hemisphere $O_H$. Since the hemispheres of $D_r$ are isomorphic, it is proven analogously that, given a 4-total coloring $c'$ provided by Lemma~\ref{l:blocks4colorable} for the blocks of the inner hemisphere, $I_H$ is 4-total colorable. Finally, since the elements of $T(C^{1}) \subset O_H$ and $T(C^{1}) \subset I_H$ are colored by Lemma~\ref{l:centrallayer4totalc}, we have that the junction of blocks $B_{i,i+1} \subset O_H$ and $B{'}_{i,i+1}, i \in \{1,3,5\}$ has no color conflicts, and it implies that $c(O_H)$ and $c'(I_H)$ are mutually compatible. Thus, it is concluded that $D_r$, with $r = 5+3k, k \in \mathbb{Z}^{+}$ is 4-total colorable. Figure~\ref{f:D54total1} displays a $4$-total coloring for an instance of $D_5$.
\end{proof}

\begin{figure}[!hbt]
    \centering
    \includegraphics[scale=0.33]{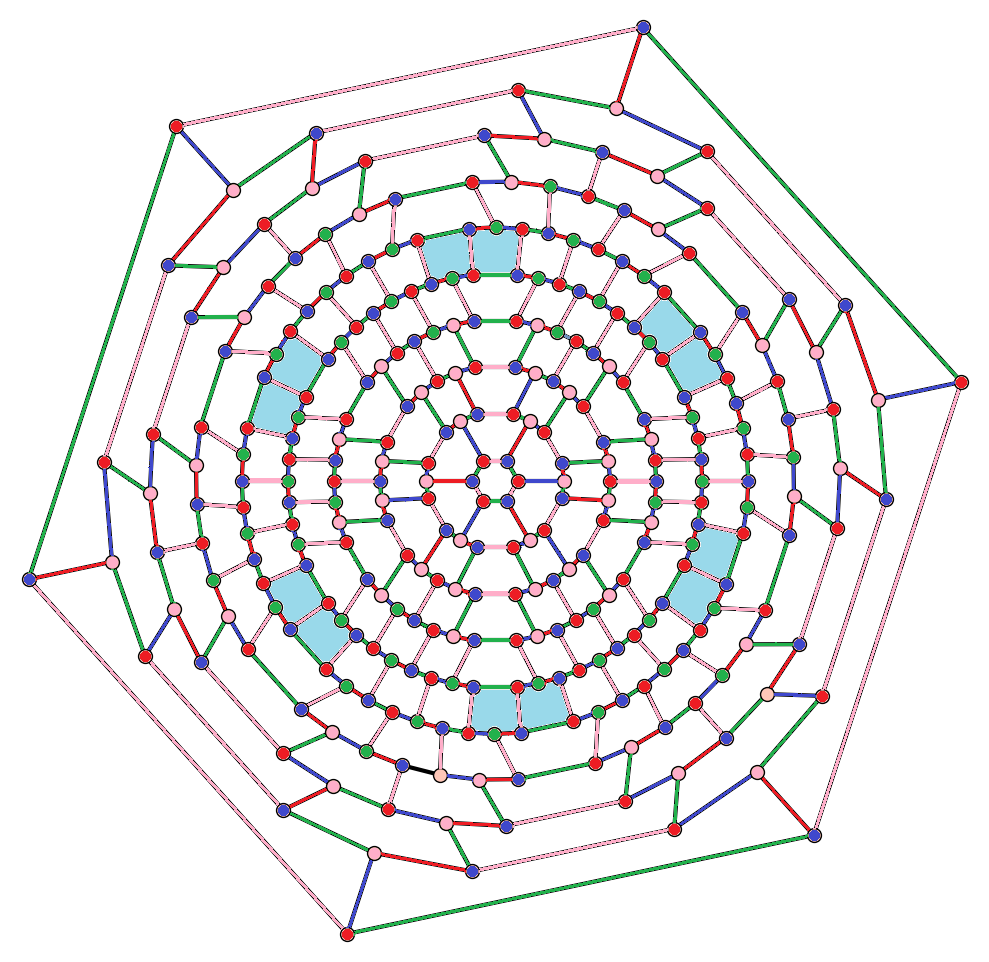}  
   \caption{A 4-total coloring of an instance of fullerene nanodisc $D_5$.}
    \label{f:D54total1}
\end{figure}
\FloatBarrier

Currently, we are trying to classify in Type 1 or Type 2 all remaining infinite families of fullerene nanodiscs. For instance, see in Figure~\ref{f:D3andD4Type1} the graphs $D_3$ and $D_4$ also colored as an extension of Lemma~\ref{l:centrallayer4totalc} and using analogous ideas of the algorithm that colors the considered infinite family of $D_r$ with $r=5+3k$. 
We believe that all fullerene nanodiscs $D_r$, $r \geq 2$ are Type 1, which brings additional positive evidence to Brinkmann, Preissmann and Sasaki's Conjecture~\ref{conjecturegirth5}. Thus, the results obtained in this paper lead us to propose the following conjecture.
\begin{conjecture}
    The fullerene nanodiscs $D_r$, $r \geq 2$ are Type 1.
\end{conjecture}    

\begin{figure}[!hbt]
    \centering
    \includegraphics[scale=0.90]{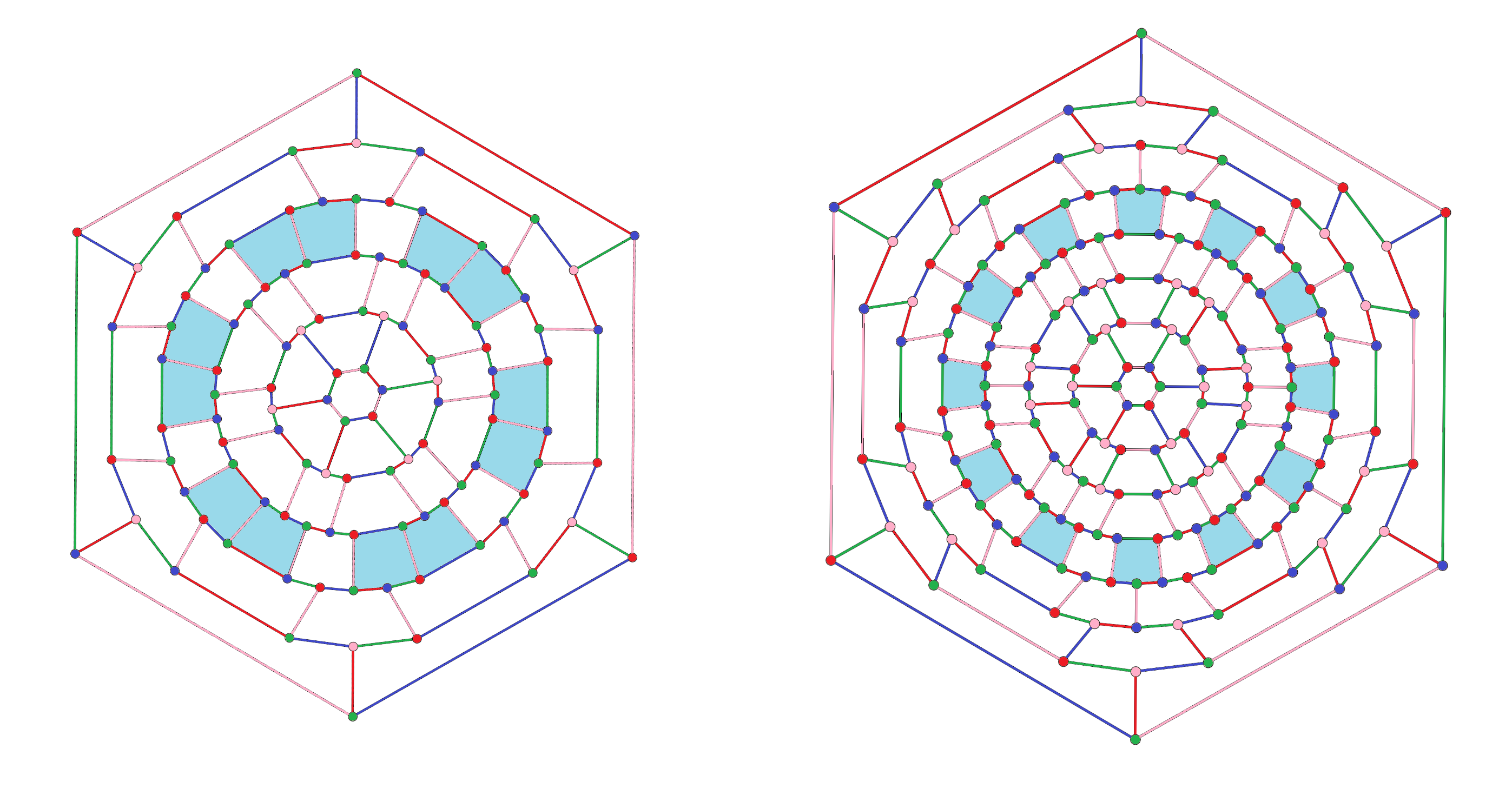}  
   \caption{A $4$-total coloring of $D_3$ and a $4$-total coloring of $D_4$, respectively.}
    \label{f:D3andD4Type1}
\end{figure}
\FloatBarrier


\bibliographystyle{abbrv}
\bibliography{bibliografia} 

\begin{thebibliography}{10}

\bibitem{andova2016}
V.~Andova, F.~Kard{\v o}s, and R.~{\v S}krekovski.
\newblock Mathematical aspects of fullerenes.
\newblock {\em Ars Mathematica Contemporanea}, 11:353–379, 2016.

\bibitem{behzad:65}
M.~Behzad.
\newblock {\em Graphs and their chromatic numbers}.
\newblock Michigan State University, 1965.

\bibitem{Brinkmann}
G.~Brinkmann, M.~Preissmann, and D.~Sasaki.
\newblock Snarks with total chromatic number 5.
\newblock {\em Discrete Mathematics and Theoretical Computer Science}, 17(1):369--382, 2015.

\bibitem{Campos2011}
C.~Campos, S.~Dantas, and C.~de~Mello.
\newblock The total-chromatic number of some families of snarks.
\newblock {\em Discrete Mathematics}, 311(12):984–988, 2011.

\bibitem{Cavicchioli2003}
A.~Cavicchioli, T.~E. Murgolo, B.~Ruini, and F.~Spaggiari.
\newblock Special classes of snarks.
\newblock {\em Acta Applicandae Mathematicae}, 76(1):57–88, 2003.

\bibitem{Dantas2023}
S.~Dantas, R.~Marinho, M.~Preissmann, and D.~Sasaki.
\newblock New results on type 2 snarks.
\newblock {\em Discussiones Mathematicae Graph Theory}, 43(4):879, 2023.

\bibitem{Gardner1976}
M.~Gardner.
\newblock Mathematical games -- snarks, boojums and other conjectures related to the four-color-map theorem.
\newblock {\em Scientific American}, 234(4):126–130, Apr. 1976.

\bibitem{c60}
H.~W. Kroto, J.~R. Heath, S.~C. O'Brien, R.~F. Curl, and R.~E. Smalley.
\newblock C60: Buckminsterfullerene.
\newblock {\em Nature}, 318(6042):162, 1985.

\bibitem{vijayaditya1971total}
N.~Vijayaditya.
\newblock On total chromatic number of a graph.
\newblock {\em Journal of the London Mathematical Society}, 2:405--408, 1971.

\bibitem{vizing:64}
V.~G. Vizing.
\newblock On an estimate of the chromatic class of a p-graph.
\newblock {\em Discret Analiz}, 3:25--30, 1964.

\bibitem{yap1996}
H.~P. Yap.
\newblock {\em Total colourings of graphs}.
\newblock Springer, 1996.

\end{thebibliography}

\end{document}